\documentclass[11pt]{amsart}
\usepackage{graphicx, amsmath, fullpage, amssymb, amsthm, amsfonts, color, algorithm}
\usepackage{enumerate}
\newtheorem{theorem}{Theorem}[section]
\newtheorem{lemma}[theorem]{Lemma}

\theoremstyle{definition}
\newtheorem{definition}[theorem]{Definition}

\newtheorem{cor}[theorem]{Corollary}
\newtheorem{remark}[theorem]{Remark}
\newtheorem{assumption}{Assumption}

\usepackage[toc,page]{appendix} 
\numberwithin{equation}{section}
\usepackage[colorlinks,
            linkcolor=red,
            anchorcolor=red,
            citecolor=blue
            ]{hyperref}

\newcommand{\E}{\mathbb E}

\newcommand{\vE}{\vec E}

\newcommand{\tr}{\textnormal{Tr}}

\renewcommand{\epsilon}{\varepsilon}

\begin{document}

\title{Extreme singular values of inhomogeneous sparse random rectangular matrices}

\author{Ioana Dumitriu}

\address{Department of Mathematics, University of California San Diego, La Jolla, CA 92093}

\email{idumitriu@ucsd.edu}

\author{Yizhe Zhu}
\address{Department of Mathematics, University of California Irvine, Irvine, CA 92697}
\email{yizhe.zhu@uci.edu}

%
\date{\today}

\begin{abstract}
We develop a unified approach to bounding the largest and smallest singular values of an inhomogeneous random rectangular matrix, based on the non-backtracking operator and the Ihara-Bass formula for general random Hermitian matrices with a bipartite block structure.  We obtain probabilistic upper (respectively, lower) bounds for the largest (respectively, smallest) singular values of a large rectangular random matrix $X$. These bounds are given in terms of the maximal and minimal $\ell_2$-norms of the rows and columns of the variance profile of $X$. The proofs involve finding probabilistic upper bounds on the spectral radius of an associated non-backtracking matrix $B$.  The two-sided bounds can be applied to the centered adjacency matrix of sparse inhomogeneous Erd\H{o}s-R\'{e}nyi bipartite graphs for a wide range of sparsity, down to criticality. In particular, for Erd\H{o}s-R\'{e}nyi bipartite graphs $\mathcal G(n,m,p)$ with $p=\omega(\log n)/n$,    and $m/n\to y \in (0,1)$, our sharp bounds imply that there are no outliers outside the support of the Mar\v{c}enko-Pastur law almost surely. This result extends the Bai-Yin theorem to sparse rectangular random matrices.
\end{abstract}

\maketitle

\section{Introduction}

\subsection{Extreme singular values of random matrices}
The asymptotic and non-asymptotic behavior of extreme singular values of random matrices is a fundamental topic in random matrix theory \cite{rudelson2010non,vershynin2010introduction}. They are crucial quantities used to provide theoretical guarantees for randomized linear algebra algorithms on large data sets, with applications in machine learning, signal processing, and data science.

Consider an $n\times m$ random matrix $X$ with $m/n\to y\in (0,1)$ with i.i.d. entries. Let $\sigma_{\max}(X)$ and $\sigma_{\min}(X)$ be the largest and smallest singular value of $X$, respectively. The classical Bai-Yin theorem \cite{bai1993limit} says that, under the finite fourth-moment assumption of the distribution of entries,  almost surely, \begin{align}\label{eq:extreme}
\frac{1}{\sqrt n} \sigma_{\max}(X)\to 1+\sqrt{y},\quad \frac{1}{\sqrt n} \sigma_{\min}(X)\to 1-\sqrt{y}.
\end{align}
This implies that there are no outliers outside the support of the Mar\v{c}enko-Pastur law for $\frac{1}{n}X^* X$.  A non-asymptotic version of the Bai-Yin theorem with a sharp constant for Gaussian matrices can be obtained from Gordon's inequality \cite{gordon1985some,vershynin2010introduction,han2022exact}; beyond the Gaussian case, similar results were given in \cite{feldheim2010universality} using the moment method for symmetric sub-Gaussian distributions (in addition, Tracy-Widom fluctuations were also proved). Under the more relaxed, finite second moment assumption, the convergence of the smallest singular value to the left edge of the Mar\v{c}enko-Pastur law was proved in \cite{tikhomirov2015limit}. Later, the convergence of the smallest and largest singular values to the edge of the spectrum was proved in \cite{chafai2018convergence,heiny2018almost}, for more general models. Finally, for sparse,  heavy-tailed random matrices, the convergence of the largest singular value was considered in \cite{auffinger2016extreme}.

Besides the sharp asymptotic behavior for the extreme singular values, non-asymptotic bounds (which do not capture the sharp constants, but the correct order) for $\sigma_{\max}(X)$ and $\sigma_{\min}(X)$ were constructed by using other arguments, including $\varepsilon$-nets
\cite{vershynin2010introduction}, matrix deviation inequalities  \cite{vershynin_2018}, and the variational principle \cite{zhivotovskiy2021dimension}. Largest singular values can also be bounded using the moment method \cite{latala2018dimension} or by using the spectral norm bound for Hermitian matrices of size $(n+m)\times (n+m)$ \cite{benaych2020spectral,bandeira2016sharp}. 

Similar results for the smallest singular values of rectangular random matrices are harder to obtain, especially when the matrices are sparse, partly because there are fewer methods of approach.  A lower bound without forth-moment assumptions was given by \cite{tikhomirov2016smallest,koltchinskii2015bounding}, and for heavy-tailed distributions in \cite{guedon2017interval,guedon2020random,tikhomirov2018sample};  None of the results above capture the sharp constant in \eqref{eq:extreme}. \cite{litvak2012smallest} considered the smallest singular values for random matrices with a prescribed pattern of zeros (which does not cover sparse Bernoulli random matrices); and \cite{gotze2022largest} considered sparse Bernoulli random matrices with $p=\omega(\log^4 n/n)$. Very recently, \cite{brailovskaya2022universality} provided a very general, non-asymptotic universality principle on the spectrum of inhomogeneous random matrices that captures the sharp constant for extreme singular values in a general setting, including inhomogeneous and sparse random rectangular matrices when $p=\omega(\log^4 n/n)$. We compare their results with ours in Remarks \ref{rmk:comparison_lowerbound}.

We should also note that the smallest singular value of a square random matrix behaves differently from the rectangular one in  \eqref{eq:extreme}, and our lower bounds on the smallest singular value (Theorems \ref{thm:expectation_A_lowerbound} and \ref{thm:least}) do not cover the square case. Specifically, in the square case, with high probability,  $\sigma_{\min}(X)$ is of order $1/\sqrt{n}$. A unified bound in both square and rectangular cases can be found in \cite{rudelson2009smallest}, which gives a lower bound $\Omega(\sqrt{n}-\sqrt{m-1})$. 
For the square random matrices,  the smallest singular value bounds were proved in \cite{livshyts2021smallest,livshyts2021smallest2,cook2018lower} for inhomogeneous random matrices, and in \cite{basak2017invertibility,basak2021sharp,che2019universality} for sparse random matrices. 

Another related topic is the study of the concentration of spectral norm for inhomogeneous random matrices, including sparse random matrices \cite{le2017concentration,benaych2020spectral,benaych2019largest,alt2021extremal,tikhomirov2021outliers}, Gaussian matrices with independent entries \cite{bandeira2016sharp,van2017spectral,latala2018dimension}, Wishart-type matrices \cite{cai2022non}, general random matrices \cite{tropp2015introduction,bandeira2021matrix,brailovskaya2022universality}, and  non-backtracking matrices \cite{benaych2020spectral,stephan2020non,bordenave2020detection}.

\subsection{Sparse random  bipartite graphs}
Extreme singular values of sparse random bipartite graphs are important quantities in the study of community detection \cite{florescu2016spectral,zhou2019analysis,wan2015class}, coding theory \cite{janwa2003tanner}, matrix completion \cite{brito2022spectral,bhojanapalli2014universal}, numerical linear algebra \cite{avron2019spectral}, and theoretical computer science \cite{deshpande2019threshold,guruswami2022l_p}.  However, classical estimates for sub-Gaussian random matrices \cite{vershynin2010introduction} cannot be directly applied to sparse random matrices due to the lack of concentration.

Recently, the non-backtracking operator has proved to be a powerful tool in the study of spectra of sparse random graphs, specifically, when the average degree of the random graph is bounded \cite{bordenave2018nonbacktracking,bordenave2020detection, bordenave2020new,  brito2022spectral,stephan2022sparse,dumitriu2021spectra} or slowly growing  \cite{benaych2020spectral,coste2021eigenvalues,wang2017limiting, stephan2020non,alt2021extremal}.  
Most results obtained with the help of non-backtracking operators are concerned with the largest eigenvalues and spectral gaps, with the exception of \cite{brito2022spectral}, which gives a lower bound on the smallest singular value of a random biregular bipartite graph, and \cite{coste2021eigenvalues}, which gives the location of isolated real eigenvalues inside the bulk of the spectrum for the non-backtracking operator. 
Lower bounds on smallest singular values for sparse random rectangular matrices were also considered in \cite{zhu2022second,guruswami2022l_p,gotze2022largest} by other methods for various models.

\subsection{Contributions of this paper}
In this paper, we provide new non-asymptotic bounds on the extreme singular values of inhomogeneous sparse rectangular matrices.  Our main tool is the non-backtracking operator for a general $n\times n$ matrix defined as follows.
\begin{definition}[Non-backtracking operator]
Let $H\in M_n(\mathbb C)$. For $e=(i,j), f=(k,l)$,
define the non-backtracking operator of $H$ as an $n^2\times n^2$ matrix $B$ such that
\begin{align}\label{eq:defB}
    B_{ef}:= H_{kl}\mathbf{1}_{j=k}\mathbf{1}_{i\not=l}.
\end{align}
\end{definition}
To associate a non-backtracking operator with a rectangular $n\times m$ random matrix $X$, we  work with the non-backtracking operator $B$ of an $(n+m)\times (n+m)$  matrix 
\begin{align}\label{eq:defH}
H=\begin{bmatrix}
 0 & X \\
 X^*  &0
\end{bmatrix}.
\end{align}
We can summarize the major steps in our proofs as follows:
\begin{enumerate}
    \item We improve the deterministic bound given in \cite{benaych2020spectral} on the largest singular value of rectangular matrices,  in terms of the spectral radius $\rho(B)$ of $B$;
    \item We provide a new deterministic  lower bound on the smallest singular value in terms of $\rho(B)$;
    \item We give an improved probabilistic bound on $\rho(B)$ for inhomogeneous random rectangular matrices;
    \item Combining the deterministic and probabilistic results, we prove, in a unified way, two-sided probabilistic bounds for a general inhomogeneous rectangular random matrix model; we also specialize them for inhomogeneous sparse random matrices for a wide range of sparsity. Our main results are stated in Section \ref{sec:main_results}.
\end{enumerate}

Although some partial efforts were previously made in the literature (\cite{auffinger2016extreme,cai2022non,gotze2022largest}) and some of our tools have been developed in
 \cite{benaych2020spectral}, the crucial contribution of this paper is a deep and unified understanding of the relationship between the singular values of a rectangular matrix and the eigenvalues of its associated non-backtracking operator, which allows us to give a unified treatment of extreme singular values. For the first time and with a minimal set of conditions, this, in turn, allows us to extend the Bai-Yin theorem to sparse random bipartite graphs (Corollary \ref{cor:no_outliers}) with average degree $\omega(\log n)$. For inhomogeneous sparse random bipartite graphs, the smallest singular value bound explicitly depends on the maximal and minimal expected degrees without extra constant factors. The upper bound on the largest singular value is valid when the maximal average degree $d=\Omega(1)$, and the lower bound on the smallest singular value is valid down to the critical regime $d=\Omega(\log n)$.

Our proof relies on the connection between extreme singular values of $X$ and the spectral radius of the corresponding non-backtracking operator $B$. In the biregular bipartite graph case \cite{brito2022spectral},  this relation is described by the Ihara-Bass formula \cite{bass1992ihara}, which results in algebraic equations involving the spectrum of $X$ and $B$. 
For inhomogeneous  Erd\H{o}s-R\'{e}nyi bipartite graphs, exact algebraic equations no longer work. Instead, we find deterministic inequalities between the extreme singular values of $X$ and the spectral radius of $B$, using a block version of the generalized Ihara-Bass formula given in  Lemma \ref{lem:block_Ihara}.

The proof of the spectral norm bound for Hermitian random matrices in \cite{benaych2020spectral}  relies on the relation between the largest eigenvalue of a  Hermitian matrix and the largest real eigenvalue of its associated non-backtracking matrix. In our case,  to get the lower bound on $\sigma_{\min}(X)$, we connected the small singular values of $X$  to the largest purely imaginary eigenvalue (in modulus) of $B$ (see Lemma \ref{lem:lowerbound}). This idea helped us establish here a similar phenomenon in a general inhomogeneous setting beyond the random bipartite biregular graph case studied in \cite{brito2022spectral}.

 {Based on \cite{benaych2020spectral},  a more refined phase transition behavior of extreme eigenvalues for homogeneous Erd\H{o}s-R\'{e}nyi graphs at $d=\log n/(\log 4-1)$  was shown in \cite{alt2021extremal}, and the same threshold was also obtained in \cite{tikhomirov2021outliers} with a different method. It is possible to combine the techniques in \cite{alt2021extremal} with our Theorems \ref{thm:main_sparse_max} and \ref{thm:expectation_A_lowerbound} to study the phase transition behavior for $d=c\log n$ in homogeneous Erd\H{o}s-R\'{e}nyi bipartite graphs, and we intend to consider this problem in subsequent work.}
 
 \subsubsection*{Linear Algebra Notation}
 We say $X\preceq Y$ for two Hermitian matrices $X$ and $Y$ if $Y-X$ is a positive semidefinite matrix. For $c\in \mathbb R$ $X\preceq c$ means $X\preceq cI$. $\|X\|$ is the spectral norm of $X$, and for a square matrix $B$, $\rho(B)$ is the spectral radius of $B$, and $\sigma(B)$ is the set of all eigenvalues of $B$.  {We denote $\sigma_{\max}(X), \sigma_{\min}(X)$ the largest and smallest singular values of a matrix $X$, respectively.}
All  $C,c, C_i, c_i$ for $i\in \mathbb N$ are universal constants. $x\vee y, x\wedge y$ are the maximum and minimum  of $x$ and $y$, respectively. Denote   $(x)_+= x$ if $x\geq 0$ and $(x)_+=0$ for $x<0$. 

\subsection*{Organization of the paper}
The rest of the paper is structured as follows. In Section \ref{sec:main_results}, we state our main results for sparse inhomogeneous Erd\H{o}s-R\'{e}nyi bipartite graphs and general rectangular random matrices. In Section \ref{sec:Ihara}, we connect the spectra of $H$ defined in \eqref{eq:defH} and that of its non-backtracking operator $B$ and prove some deterministic bounds on the extreme singular values of $X$. In Sections \ref{sec:bound_rho_B}  and \ref{sec:trace}, we give a probabilistic upper bound on $\rho(B)$ for a general random matrix  $X$ with the moment method. In Sections \ref{sec:prob_largest} and \ref{sec:prob_smallest}, we give proofs for the probabilistic bounds on extreme singular values of $X$ and specialize them for inhomogeneous sparse bipartite random graphs.

\section{Main results}\label{sec:main_results}


\subsection{Inhomogeneous  Erd\H{o}s-R\'{e}nyi  bipartite graphs}\label{sec:sparse}

\begin{definition}[Inhomogeneous  Erd\H{o}s-R\'{e}nyi bipartite graph]
An \textit{inhomogeneous Erd\H{o}s-R\'{e}nyi bipartite graph} $G\sim \mathcal G(n,m,p_{ij})$ is a random bipartite graph defined on a vertex set $V=V_1\cup V_2$, where $|V_1|=n, |V_2|=m$ such that an edge $ij$, $i\in [n], j\in [m]$ is included independently with probability $p_{ij}$.
Let $A\in \{0,1\}^{n\times m}$ be the \textit{biadjacency matrix} of $G$ such that $A_{ij}=1$ if $ij$ is an edge in $G$ and $A_{ij}=0$ otherwise.  The \textit{adjacency matrix} of $G$ is given by $\begin{bmatrix}
 0 & A\\
 A^\top & 0
\end{bmatrix}.$
\end{definition}

\subsubsection*{Notation and assumptions}

 {In Section \ref{sec:sparse}, we will use the following notation and assumptions and specify occasional, limited-use additional assumptions as necessary.
\begin{itemize}
\item The maximal expected degree among all vertices from $V_1\cup V_2$ is denoted by
    \[d:= \max_{i\in [n],j\in [m]} \left(  \sum_{k=1}^m p_{ik}, \sum_{k=1}^n p_{kj}\right)\] 
\item  The normalized maximal expected degrees from $V_2$ (respectively, $V_1$) as
\begin{align}
    \rho_{\max}:&=\frac{1}{d}\max_{ {i\in [n]}} \sum_{ {j\in [m]}}p_{ij}, \quad 
    \tilde{\rho}_{\max}:=\frac{1}{d}\max_{j\in [m]}\sum_{i\in [n]}p_{ij},
\end{align}
and the normalized minimal expected degree in $V_2$ is defined as 
\[ \tilde{\rho}_{\min}= {\frac{1}{d}}\min_{j\in [m]}\sum_{i\in [n]} p_{ij}(1-p_{ij}).\] 
\item Denote $N:=n\vee m$  and  {$\eta:=\sqrt{\log N/d}$}. 
\end{itemize}
\begin{assumption}\label{assumption_1}
We let $\gamma :=\rho_{\max}\wedge \tilde \rho_{\max}$, where $\gamma\in (0,1]$ and $\gamma=\Omega(1)$. Note that when all $p_{ij}=p$, $ {\gamma=\frac{n\wedge m}{n \vee m}}$ is the aspect ratio. Assume  $d\geq \gamma^{-1/2}$ and $d^{\frac{7}{2}}\max_{ij}p_{ij}\leq N^{-\frac{1}{10}}$.
\end{assumption}
}


We first state the following upper bound on the largest singular values of $A-\mathbb EA$ for a wide range of sparsity down to $d=\Omega(1)$.
 {\begin{theorem}[Largest singular value for inhomogeneous sparse random matrices]\label{thm:main_sparse_max}
Let $A$ be the biadjacency matrix of an inhomogeneous random graph. Under Assumption \ref{assumption_1},
\begin{align}\label{eq:expectation_A}
    \frac{1}{\sqrt d}\mathbb E [\sigma_{\max}(A-\mathbb E A)]\leq 1+\sqrt{\gamma} +O\left({\frac{\eta}{1\vee \sqrt{\log \eta }}}\right).
\end{align}
Moreover,  with probability at least $1-O(N^{-3}),$  we have
\begin{align}\label{eq:concentration_A}
   \frac{1}{\sqrt{d}}\sigma_{\max}(A-\mathbb E A)\leq 1+\sqrt{\gamma} +O\left(\frac{\eta}{1\vee \sqrt{\log \eta }}\right).
\end{align}
\end{theorem}}

\begin{remark} \label{rmk:Knowles}
 {
Below, we qualitatively explain Theorem \ref{thm:main_sparse_max} in the various $d$ regimes. }
   When $d=\omega(\log N)$ and $d\leq N^a$ for some constant $a>0$, \eqref{eq:expectation_A} is dominated by the first term, which gives $1+\sqrt{\gamma}+o(1)$. This is sharp, and it recovers the results of \cite[Theorem 4.9 and Example 4.10]{latala2018dimension},  {and \cite[Theorem 3.5]{cai2022non}}. Moreover,
our model can be seen as a specific case considered in \cite{benaych2019largest}, but their bound yields the weaker $2+o(1)$ on the right-hand side of \eqref{eq:expectation_A}.  The fact that we cannot cover denser regimes is an artifact of our proof method (see Theorem \ref{thm:rhoB}).

 When $d = O(\log N)$, Theorem \ref{thm:main_sparse_max} is optimal up to a constant factor. 
When $d=o(\log N)$, the second term in \eqref{eq:expectation_A} is dominating. Our results yield the sharp bound $O ({\eta/\sqrt{\log \eta}} )$. This is tight up to a constant factor, down to $d\geq \gamma^{-\frac{1}{2}}$, and it agrees with the results in  \cite{benaych2019largest,krivelevich2003largest}  for non-bipartite graphs.  {Note that the results in \cite{latala2018dimension} and \cite{cai2022non} both imply an $O(\eta)$ upper bound, which is strictly weaker.}
\end{remark}

The real novelty arises in the following theorem, which provides the smallest singular value bound on $A-\mathbb EA$, down to $d=\Omega(\log n)$,  and it is sharp in certain cases.  Our sparsity assumption on $d$ is optimal up to a constant factor: when $d< (1-\varepsilon)\log n$, with high probability there are isolated vertices, and that implies $\sigma_{\min}(A-\mathbb EA)=0$.

For the next result, we need the following additional assumption.

 {
\begin{assumption}\label{assumption_2}
    Let $n\geq  m$ and $\delta\in (0,1)$. Assume $\tilde{\rho}_{\max}\geq \rho_{\max} $,  $\tilde{\rho}_{\min}> \sqrt{\gamma}$, and $\min\{ 1-\sqrt{\gamma}, \tilde{\rho}_{\min}- \sqrt{\gamma}\}=\Omega(1)$, $d^{\frac{7}{2}}\max_{ij} p_{ij}\leq  n^{-\frac{1}{10}}$, and there exists an absolute constant $C>0$ such that
\[d\geq C \max \left\{ \delta^{-1/2}, \delta^{-2} \gamma^{-1} \log n \right\}. \] 
\end{assumption}}

 {\begin{theorem}[Smallest singular value for inhomogeneous sparse random matrices] \label{thm:expectation_A_lowerbound} Under Assumptions \ref{assumption_1} and \ref{assumption_2},  
\begin{align} \label{eq:sminE}
\frac{1}{\sqrt d}\mathbb E \sigma_{\min} (A-\mathbb EA)\geq \sqrt{(1-\sqrt{\gamma})(\tilde \rho_{\min}-\sqrt{\gamma})}-O(\delta^{1/4}).
\end{align}
Moreover, with probability at least $1-O(n^{-3})$,
\begin{align}\label{eq:smin}
 \frac{1}{\sqrt{d} }\sigma_{\min}(A-\mathbb EA )&\geq \sqrt{(1-\sqrt{\gamma})(\tilde \rho_{\min}-\sqrt{\gamma})}-O(\delta^{1/4}).
\end{align}
\end{theorem}}


\begin{remark}\label{rmk:comparison_lowerbound}
The very general universality principle proved in \cite[Theorem 2.13]{brailovskaya2022universality} 
implies 
\begin{align}\label{eq:vanHandel}
d^{-1/2}\sigma_{\min}(A-\mathbb EA)\geq \sqrt{\tilde \rho_{\min}}-\sqrt{\gamma}-O(d^{-1/6} \log^{2/3}(n))
\end{align} with high probability. The leading constant  $\sqrt{\tilde \rho_{\min}}-\sqrt{\gamma}$ in \eqref{eq:vanHandel} is strictly better than our leading constant $\sqrt{(1-\sqrt{\gamma})(\tilde{\rho}_{\min}-\sqrt{\gamma})}$ in  \eqref{eq:smin} for inhomogeneous random graphs and it matches our leading constant in the homogeneous case when  $\tilde{\rho}_{\min}=1$. 
However, their results are only valid for the regime when $d = \omega(\log^4 n)$---the benefit of our method is that our results extend down to $d=\Omega(\log n)$.

 {\cite{cai2022non} studied concentration inequalities for inhomogeneous Wishart-type matrices.  When $d=\omega(\log n)$, by triangle inequality, \cite[Theorem 3.5]{cai2022non}  implies
$d^{-1}\mathbb E[\sigma_{\min}^2(A-\mathbb EA)]\geq \tilde{\rho}_{\min}-2\sqrt{\gamma}-\gamma-O(\sqrt{\eta}).$
Their leading constant is strictly weaker than ours in all regimes.}
In \cite{gotze2022largest}, the authors considered the smallest singular value of a sparse random matrix $X=A\circ Y$, where $A$ is a sparse Bernoulli matrix, and $Y$ is a Wigner matrix whose entries have bounded support, and $\circ$ is the Hadamard product. This covers the homogeneous Erd\H{o}s-R\'{e}nyi case, but not the inhomogeneous one. The authors showed that the smallest singular value is   $\Omega(\sqrt{np})$ when $p=\omega(\log^4 n)$ (see \cite[Theorem 1.2]{gotze2022largest}). Note that our Theorem \ref{thm:least} covers a more general inhomogeneous model, which includes  \cite{gotze2022largest}.
\end{remark}

\begin{remark}
     {The lower bound in  \eqref{eq:smin} is $\sqrt{(1-\sqrt\gamma)(\tilde{\rho}_{\min}-\sqrt{\gamma})}-o(1)$ when $d=\omega(\log n)$. When $d = \Omega(\log n)$, we can obtain an $\Omega(1)$ lower bound when $\delta$ is sufficiently small. We have assumed $\gamma=\Omega(1)$ for simplicity throughout. However, Theorems \ref{thm:main_sparse_max} and \ref{thm:expectation_A_lowerbound} also work for $\gamma=o(1)$, see Sections \ref{sec:prob_largest} and \ref{sec:prob_smallest} with weaker bounds. It remains an open question to find the optimal dependence on $\gamma$ when $\gamma=o(1)$ in  \eqref{eq:concentration_A} and \eqref{eq:smin}.}
\end{remark}

For the next result, we assume $d=\omega(\log n)$, and the expected degree of each vertex concentrates. The Mar\v{c}enko-Pastur law for the matrix $\frac{1}{d} (A-\mathbb EA)^\top (A-\mathbb EA)$ can be proved in the same way the semicircle law was proved in \cite[Corollary 4.3]{zhu2020graphon} via graphon theory. This implies the upper bound and lower bound given by  {\eqref{eq:concentration_A}} and  {\eqref{eq:smin}} are tight. 
We state the generalization of the  Bai-Yin theorem   to sparse random bipartite graphs in the following corollary.

\begin{cor}[Bai-Yin theorem for supercritical random  bipartite graphs] \label{cor:no_outliers} Let $A$ be the biadjacency matrix of an inhomogeneous Erd\H{o}s-R\'{e}nyi  bipartite graph sampled from $\mathcal G(n,m,p_{ij})$.
Assume $\frac{m}{n}\to y \in (0,1)$ as $n\to\infty$ and   the parameters $\{p_{ij}\}$ satisfy
\begin{align}\label{eq:homo_assumption}
\max_{i\in [n]}\left|\frac{1}{d}\sum_{j\in [m]} p_{ij}-1 \right|=o(1),  \quad 
\max_{j\in [m]}\left|\frac{1}{d}\sum_{i\in [n]} p_{ij}-y \right|=o(1).
\end{align} 
Assume $d=\omega(\log n), d\leq n^{1/5}$ and $\max_{ij} p_{ij}=O(d/n)$.
Then, almost surely,  
\begin{align}
\lim_{n\to\infty} \frac{1}{\sqrt{d}}\sigma_{\min}(A-\mathbb E A )= 1-\sqrt{y}, \quad \quad 
\lim_{n\to\infty} \frac{1}{\sqrt{d}} \sigma_{\max}(A-\mathbb E A)= 1+\sqrt{y}.\label{eq:upper_Bai}
\end{align}
\end{cor}

Corollary \ref{cor:no_outliers} covers the homogeneous Erd\H{o}s-R\'{e}nyi  bipartite graph $\mathcal G(n,m,p)$ with $p=\omega(\log n/n)$ and $p\leq  n^{-4/5}$. 
For denser cases when $p\geq n^{-4/5}$, the right edge limit \eqref{eq:upper_Bai} has been obtained in \cite{latala2018dimension}, and the left edge limit in \eqref{eq:upper_Bai} may be obtainable by moment methods without using the non-backtracking operator \cite{bai1993limit,feldheim2010universality}.

\subsection{Inhomogeneous random rectangular matrices}\label{sec:general}
To prove our results in Section \ref{sec:sparse}, we work with a more general matrix model whose entries have bounded support. Theorem \ref{thm:main_sparse_max} and Theorem \ref{thm:expectation_A_lowerbound} are obtained from Theorem \ref{thm:lambdamax} and Theorem \ref{thm:least} by specifying model parameters and taking $q=\sqrt{d}$.

 {\subsubsection*{Notation and assumptions}
In Section~\ref{sec:general}, we use the following notation and assumptions:
\begin{itemize}
\item The minimal column sum of variances is denoted by
\[\tilde{\rho}_{\min}:=\min_{j\in [m]}\sum_{i\in [n]} \mathbb E |X_{ij}|^2. \]
\item Denote $N:=n\vee m$ and $\eta:=\sqrt{\log N}/{q}$.
\end{itemize}
\begin{assumption}\label{assumption_3}
    We assume there exist   $q>0, \kappa \geq 1$, {$\rho_{\max}, \tilde{\rho}_{\max}$} such that  \begin{align}\label{eq:Psigma}
    \max_{ij} |X_{ij}| &\leq \frac{1}{q},\quad \max_{ij}\E |X_{ij}|^2 \leq \frac{\kappa}{N},\\
      \max_{{j}\in [m]} \sum_{{i}\in [n]} \mathbb E |X_{ij}|^2 &\leq  \rho_{\max}, \quad 
    \max_{i\in [n]}  \sum_{j\in [m]} \mathbb E  |X_{ij}|^2 \leq  \tilde{\rho}_{\max},\label{eq:Psigma2}
\end{align}
where $\rho_{\max}\vee \tilde{\rho}_{\max}=1$ and $
\rho_{\max}\wedge \tilde{\rho}_{\max}=\gamma\in (0,1]$, and $\gamma=\Omega(1)$. Assume  \[
    \gamma^{-\frac{1}{4}}\leq q\leq N^{\frac{1}{10}}\kappa^{-\frac{1}{9}}\gamma^{-\frac{1}{18}} .\]
\end{assumption}}

 {\begin{theorem}[Largest singular value]\label{thm:lambdamax}
Let $X$ be an $n\times m$ random matrix with independent entries, and $\mathbb EX=0$.  Then under Assumption~\ref{assumption_3},
\begin{align}\label{eq:Esigma_max}
    \mathbb E[ \sigma_{\max}(X)]\leq \sqrt{\gamma}+1+ O\left(\frac{\eta}{{1\vee \sqrt{\log \eta}} }\right).
\end{align}
 Moreover, with probability at least $1-O(N^{-3})$,
\begin{align}\label{eq:sigmax_X}
    \sigma_{\max}(X)\leq  \sqrt{\gamma} +1+O\left(\frac{\eta}{1\vee \sqrt{\log \eta }}\right).  
\end{align} 
\end{theorem}}

 {\begin{assumption}\label{assumption_4}
    Let $n\geq  m$ and $\delta\in (0,1)$. Assume $\tilde{\rho}_{\max}\geq \rho_{\max} $,  $\tilde{\rho}_{\min}> \sqrt{\gamma}$, and $\min\{ 1-\sqrt{\gamma}, \tilde{\rho}_{\min}- \sqrt{\gamma}\}=\Omega(1)$.
Assume for  $\delta\in [0,1)$, $q$ satisfies
\[ C\max \left\{ \delta^{-1/2}, \delta^{-1} \gamma^{-1/2}\sqrt{\log n} \right\}\leq q\leq  n^{\frac{1}{10}}\kappa^{-\frac{1}{9}}\gamma^{-\frac{1}{18}}\]
for an absolute constant $C>0$.  
\end{assumption}}

 {
\begin{theorem}[Smallest singular value] \label{thm:least}
Let $X$ be an $n\times m$ random matrix with independent entries and $\mathbb E X=0$. Under Assumptions~\ref{assumption_3} and \ref{assumption_4}, we have
\begin{align}\label{eq:Esigma_min}
    \mathbb E[\sigma_{\min}(X)]\geq \sqrt{(1-\sqrt{\gamma})(\tilde{\rho}_{\min}-\sqrt{\gamma})}-O(\delta^{1/4}) .
\end{align}
Moreover, with probability at least $1-O(n^{-3})$, we have 
\begin{align}\label{eq:sigma_min}
 \sigma_{\min}(X)&\geq \sqrt{(1-\sqrt{\gamma})(\tilde{\rho}_{\min}-\sqrt{\gamma})}-O(\delta^{1/4}).
\end{align}
\end{theorem}}

 \begin{remark}
 Through standard truncation arguments, described in detail, e.g., in \cite[Pages 41-73]{borodin2019random} and \cite[Chapter 5]{bai2010spectral},  our results in Section \ref{sec:general} can be applied to random variables with unbounded support. This includes, for example, the dense Gaussian case and any other cases satisfying certain Lindeberg's conditions \cite{bai2010spectral}.
 \end{remark}

\section{Spectral relation between $X$ and $B$}\label{sec:Ihara}

\subsection{Generalized Ihara-Bass formula}\label{sec:Genarlized_Ihara}

We will make use of the following generalized Ihara-Bass formula proved in \cite{benaych2020spectral,watanabe2009graph}. When $H$ is the adjacency matrix of a graph, Lemma \ref{lem:IBformula} reduces to the classical Ihara-Bass formula in \cite{bass1992ihara,kotani2000zeta}.
\begin{lemma}[Lemma 4.1 in \cite{benaych2020spectral}]\label{lem:IBformula}
Let $H\in M_n(\mathbb C)$ with associated non-backtracking matrix $B$. Let $\lambda\in \mathbb C$ satisfying $\lambda^2\not=H_{ij}H_{ji}$ for all $i,j\in [n]$. Define $H(\lambda)$ and $M(\lambda)=\textnormal{diag}(m_i(\lambda))_{i\in [n]}$ as
\begin{align}
    H_{ij}(\lambda)&:=\frac{\lambda H_{ij}}{\lambda^2-H_{ij}H_{ji}},\quad 
    m_i(\lambda):=1+\sum_{k\in [n]}\frac{H_{ik}H_{ki}}{\lambda^2-H_{ik}H_{ki}}. \label{eq:defmi}
\end{align}
Then $\lambda\in \sigma(B)$ if and only if $\det (M(\lambda)-H(\lambda))=0.$
\end{lemma}

By itself, Lemma \ref{lem:IBformula} is not sharp enough to yield a tight upper bound for $\sigma_{\max}(X)$, and it cannot yield any results for the smallest singular values. 
Therefore, we have developed a customized approach for the block matrix model, including a sharp analysis of the non-backtracking operator, which will lead to tight results in both cases. The first step in this approach is the following customized version of  
Lemma \ref{lem:IBformula}.

\begin{lemma}\label{lem:block_Ihara}
Let $X$ be an $n\times m$ complex matrix and $H=\begin{bmatrix}
0 & X\\
X^* &0 
\end{bmatrix}$.
Let $B$ be the non-backtracking operator associated with $H$. Define an $n\times m$ matrix $X(\lambda)$, and two diagonal matrices $M_1(\lambda)=\textnormal{diag}(m^{(1)}_i(\lambda))_{i\in [n]}$,  $M_2(\lambda)=\textnormal{diag}(m^{(2)}_i(\lambda))_{i\in [m]}$ as follows:
\begin{align*}
    X_{ij}(\lambda)=\frac{\lambda X_{ij}}{\lambda^2-|X_{ij}|^2},\quad 
    m^{(1)}_i(\lambda)=1+\sum_{k\in [m]}\frac{|X_{ik}|^2}{\lambda^2-|X_{ik}|^2}, \quad 
     m^{(2)}_j(\lambda)=1+\sum_{k\in [n]}\frac{|X_{kj}|^2}{\lambda^2-|X_{kj}|^2}.
\end{align*}
Assume $M_1(\lambda)$ is non-singular.
Then $\lambda\in \sigma (B)$ if and only if 
\begin{align}\label{eq:IB-bipartite}
    \det(M_2(\lambda)-X^*(\lambda) M_1(\lambda)^{-1} X(\lambda)) \det (M_1(\lambda))=0.
\end{align}
\end{lemma}
\begin{proof}
From Lemma \ref{lem:IBformula}, $\lambda\in \sigma(B)$ if and only if 
$ 
    \det\begin{bmatrix}
     M_1(\lambda) & -X(\lambda)\\
     -X^*(\lambda)  &M_2(\lambda)
    \end{bmatrix}=0. \notag 
$
Since $M_1(\lambda)$ is non-singular, by the determinant formula for block matrices, \eqref{eq:IB-bipartite} holds.
\end{proof}

\subsection{Deterministic upper bound on the largest singular value}

Using Lemma \ref{lem:IBformula}, we bound $\sigma_{\max}(X)$ in terms of the maximal Euclidean norm of rows and columns of $X$ and  $\rho(B)$ as follows.

\begin{lemma}\label{lem:upperbound}
Let $H,X$, and $B$ be defined as in Lemma \ref{lem:block_Ihara}. Suppose    
\begin{align}
    \max_{ij} |X_{ij}| &\leq \delta, \quad 
    \max_{j\in [m]} \sum_{i\in [n]} |X_{ij}|^2 \leq \tilde{\rho}_{\max}(1+\delta),\label{eq:bound1}\\
    \max_{i\in [ {n}]} \sum_{j\in [ {m}]} |X_{ij}|^2 &\leq \rho_{\max}(1+\delta),\label{eq:bound2}
\end{align}
with $\rho_{\max}\vee \tilde{\rho}_{\max}=1, \rho_{\max}\wedge \tilde{\rho}_{\max}=\gamma$  and $\delta\in [0,\gamma^{\frac{1}{2}}]$.
Let  $\lambda\geq \max \{\gamma^{\frac{1}{4}}(1+\sqrt \delta), \rho(B)\}.$
Then 
\begin{align}
   \sigma_{\max}(X)^2\leq \left(\lambda +\frac{1}{ \lambda} \right)\left( \lambda +\frac{\gamma}{ \lambda} \right)+6\gamma^{-1}\delta \left(2\lambda+\frac{1+\gamma}{\lambda}\right)+36\gamma^{-2}\delta^2. \notag 
\end{align}
\end{lemma}
\begin{proof}
From Lemma \ref{lem:IBformula}, by continuity,
$M(\lambda)-H(\lambda)\succeq 0$ for 
$\lambda\geq  \lambda_0:=\max \{\gamma^{\frac{1}{4}}(1+\sqrt{\delta}), \rho(B)\}$.
For any $\lambda\geq \lambda_0$, we have
\begin{align}\label{eq:lambdaH}
    |\lambda H_{ij}(\lambda)-H_{ij}|=\frac{|H_{ij}|^3}{\lambda^2-|H_{ij}|^2}\leq \frac{ \delta|H_{ij}|^2}{(\lambda^2-\delta^2)}\leq \gamma^{-\frac{1}{2}} \delta |H_{ij}|^2.
\end{align}
By Gershgorin circle theorem, it implies that
\begin{align}\label{eq:Gcircle}
\| \lambda H(\lambda)-H\|_2\leq \gamma^{-\frac{1}{2}} \delta \max_i \sum_{j} |H_{ij}|^2\leq   \gamma^{-\frac{1}{2}}\delta(1+\delta)\leq 2\gamma^{-\frac{1}{2}}\delta.
\end{align}
For any $i\in V_1$, from \eqref{eq:defmi}, for any $\lambda\geq \lambda_0$,
\begin{align}
    \lambda m_i(\lambda)-\left(\lambda+\frac{\rho_{\max}}{\lambda}\right)&=\sum_{k\in V_2} \frac{\lambda |H_{ik}|^2}{\lambda^2-|H_{ik}|^2}-\frac{\rho_{\max}}{\lambda} \leq \frac{\rho_{\max}}{\lambda}\left( \frac{\lambda^2(1+\delta)}{\lambda^2-\delta^2}-1 \right)=\rho_{\max} \delta \frac{\lambda(1+\lambda^{-2}\delta)}{\lambda^2-\delta^2} \notag\\
    &\leq \rho_{\max}\delta \frac{\lambda(1+\gamma^{-\frac{1}{2}})}{\lambda^2-\delta^2}\leq 2\rho_{\max}\delta(\gamma^{-\frac{1}{2}}+\gamma^{-1})\leq 4\gamma^{-1}\delta, \label{eq:mV1}
\end{align}
where in the last step we consider the cases $\lambda\geq 2$ and $\lambda<2$ and use the inequality $\lambda^2-\delta^2\geq \gamma^{\frac{1}{2}}$ in the second case. 
For any $i\in V_2$, similarly,
\begin{align}\label{eq:mV2}
     \lambda m_i(\lambda)-\left(\lambda+\frac{\tilde{\rho}_{\max}}{\lambda}\right)&=\sum_{k\in V_1} \frac{\lambda |H_{ik}|^2}{\lambda^2-|H_{ik}|^2}-\frac{\tilde{\rho}_{\max}}{\lambda}\leq \frac{\tilde{\rho}_{\max}}{\lambda} \left( \frac{1+\delta}{1-\lambda^{-2}\delta^2}-1\right)\leq  4\gamma^{-1}\delta.
\end{align}

Then for $\lambda\geq \lambda_0$,  with \eqref{eq:Gcircle}, \eqref{eq:mV1}, and \eqref{eq:mV2},
\begin{align}\label{eq:PSDH}
   0 \preceq \lambda (M(\lambda)-H(\lambda))\preceq \begin{bmatrix}
    \lambda +\frac{\rho_{\max}}{ \lambda} +6 \gamma^{-1}\delta &0\\
    0 &  \lambda +\frac{\tilde{\rho}_{\max}}{\lambda} + 6 \gamma^{-1}\delta
    \end{bmatrix}-H.
\end{align}

Let $
d_1:= \lambda +\frac{\rho_{\max}}{ \lambda} +6 \gamma^{-1}\delta,  d_2:=\lambda +\frac{\tilde{\rho}_{\max}}{\lambda} + 6 \gamma^{-1}\delta.$
Then from \eqref{eq:PSDH}, $
  \Delta= \begin{bmatrix}
      d_1 I & -X\\
        -X^* & d_2 I
    \end{bmatrix} \succeq 0$.
Note that the following matrix factorization holds:
\begin{align*}
\begin{bmatrix}
d_1 I &-X\\
-X^* & d_2 I
\end{bmatrix}=\begin{bmatrix}
\sqrt{d_1} I & 0\\
0& \sqrt{d_2} I 
\end{bmatrix}\cdot\begin{bmatrix}
I &-(d_1d_2)^{-\frac{1}{2}}X\\
-(d_1d_2)^{-\frac{1}{2}}X^* & I
\end{bmatrix}\cdot \begin{bmatrix}
\sqrt{d_1} I & 0\\
0& \sqrt{d_2}  I
\end{bmatrix}.
\end{align*}
 Therefore, on the right-hand side of the above equation, the second matrix is positive semidefinite. We obtain
\begin{align*}
    X^*X\preceq d_1d_2&= \left(\lambda +\frac{1}{ \lambda}+6\gamma^{-1}\delta\right)\left( \lambda +\frac{\gamma}{ \lambda}+6\gamma^{-1}\delta   \right)\\
    &= \left(\lambda +\frac{1}{ \lambda} \right)\left( \lambda +\frac{\gamma}{ \lambda} \right)+6\gamma^{-1}\delta \left(2\lambda+\frac{1+\gamma}{\lambda}\right)+36\gamma^{-2}\delta^2.
\end{align*}
This gives the desired upper bound on $\sigma_{\max}(X)$.
\end{proof}

Lemma \ref{lem:upperbound} gives an upper bound on $\sigma_{\max}(X)$ when the parameter $\delta\in [0,\gamma^{\frac{1}{2}}]$. By rescaling the entries in $H$, we can  obtain a general  bound  depending  on the following quantities:
 \[\|H\|_{1,\infty}=\max_{ij}|H_{ij}|, \quad \|H\|_{2,\infty}:=\max_i\left(\sum_{j} |H_{ij}|^2\right)^{\frac{1}{2}} ~,\]  without the restriction on the range of $\delta$. This is more convenient for us to handle sparse random bipartite graphs in the critical and subcritical regimes. Define 
 \[f(x)=\begin{cases}
(x+x^{-1})\left(x+\frac{\gamma}{x}\right), & x\geq \gamma^{\frac{1}{4}} \\
(\gamma^{\frac{1}{2}}+1)^2 , & 0\leq x\leq \gamma^{\frac{1}{4}}
\end{cases},\quad g(x)=\begin{cases}2\left(x+\frac{1}{x}\right), &  x\geq \gamma^{\frac{1}{4}}\\
4& 0\leq x\leq \gamma^{\frac{1}{4}}
\end{cases},\] 
where $\gamma\in (0,1]$ is a constant such that \begin{align}\label{eq:XX*assumption}
    \|X\|_{2,\infty}\wedge \|X^*\|_{2,\infty}&\leq  \gamma^{\frac{1}{2}}\|H\|_{2,\infty}.
\end{align}

\begin{lemma} [Deterministic upper bound on $\sigma_{\max}(X)$] Let $X$ be an $n\times m$ matrix, $H=\begin{bmatrix}
 0 & X \\
 X^*  &0
\end{bmatrix}$, and $\gamma$ be the constant in \eqref{eq:XX*assumption}.
The following inequality holds:
\begin{align}\label{eq:CorK}
\sigma^2_{\max}(X)\leq  &\|H\|_{2,\infty}^2~ f\left( \frac{\rho(B)}{ \|H\|_{2,\infty}}\right)+12\gamma^{-\frac{5}{4}} g\left(\frac{\rho(B)}{ \|H\|_{2,\infty}}\right) \|H\|_{2,\infty} \|H\|_{1,\infty}+36\gamma^{-2}\|H\|_{1,\infty}^2.  
\end{align}
\end{lemma}

\begin{proof}
First assume  $ \|H\|_{2,\infty}=1$.
Set  $\delta=\|H\|_{1,\infty}$. From \eqref{eq:XX*assumption}, $\delta\leq \gamma^{\frac{1}{2}}$. By Lemma \ref{lem:upperbound}, for $\lambda_0=\max \{\gamma^{\frac{1}{4}}(1+\sqrt{\delta}), \rho(B)\}$,
we have 
\begin{align}\label{eq:lambda_0_bound}
    \sigma^2_{\max}(X)\leq f(\lambda_0)+6\gamma^{-1}\delta \left(2\lambda_0+\frac{1+\gamma}{\lambda_0}\right)+36\gamma^{-2}\delta^2.
\end{align}
When $\lambda_0=\rho(B)$, it implies
\begin{align}
    \sigma^2_{\max}(X)\leq f(\rho(B))+6\gamma^{-1}\delta g(\rho(B))+36\gamma^{-2}\delta^2.\label{eq:lambda_0_case1}
\end{align}
    When $\lambda_0=\gamma^{\frac{1}{4}}(1+\sqrt{\delta})$, from \eqref{eq:lambda_0_bound}, 
\begin{align}
     \sigma_{\max}^2(X)&\leq \lambda_0^2+1+\gamma+\frac{\gamma }{\lambda_0^2} +6\gamma^{-1}\delta\left( 2\lambda_0+\frac{1+\gamma}{\lambda_0}\right)+36\gamma^{-2}\delta^2 \notag\\
     &\leq (\sqrt{\gamma}+1)^2+45\gamma^{-\frac{5}{4}}\delta +36 \gamma^{-2}\delta^2 \leq f(\rho(B))+12 \gamma^{-\frac{5}{4}}\delta g(\rho(B))+ 36\gamma^{-2}\delta^2. \label{eq:lambda_0_case2}
\end{align}
Combining \eqref{eq:lambda_0_case1} and \eqref{eq:lambda_0_case2}, we have 
\begin{align}
    \sigma_{\max}(X)^2\leq f(\rho(B))+12 \gamma^{-\frac{5}{4}} g(\rho(B))\|H\|_{1,\infty}+ 36 \gamma^{-2}\|H\|_{1,\infty}^2. \notag
\end{align}
Then, for general $H$, by considering $\frac{H}{ \|H\|_{2,\infty}}$ and repeating the proof above, we get the desired bound.
\end{proof}

\subsection{Deterministic lower bound on the smallest singular value}

The following lemma gives us a connection between the spectral radius of $B$ and the smallest singular value of $X$. The proof relies on finding a relation between purely imaginary eigenvalues of  $B$ and singular values of $X$.

\begin{lemma}[Deterministic lower bound on $\sigma_{\min}(X)$]\label{lem:lowerbound}
Let $H,X$ and $B$ be defined as in Lemma \ref{lem:block_Ihara} and $n\geq m$.  Let $0<\gamma< 1$,  $\delta\in [0,1)$,  $C_1>0$ such that 
\begin{align*}
    \max_{ij} |X_{ij}| &\leq \delta,  \quad \|X\|_2\leq C_1, \\
  \max_{i\in [n]}\sum_{j\in [m]} |X_{ij}|^2 &\leq  \gamma(1+\delta), \quad 
   \tilde{\rho}_{\min}(1-\delta) \leq \sum_{i\in [n]} |X_{ij}|^2 \leq 1+\delta, \quad \forall j\in [m].
\end{align*}
Define $\beta_0=\max \{ \gamma^{\frac{1}{4}}(1+\sqrt{\delta}),\rho(B)\}.$ Then  for $\beta\geq \beta_0$,
\begin{align}\label{eq:XXlowerbound}
   \sigma_{\min}^2(X)\geq  \frac{\sqrt{\gamma}- \gamma}{\sqrt{\gamma}+\delta}\left( \frac{\beta^2}{
    \beta^2+\delta^2}\tilde{\rho}_{\min}-\beta^2-C_{\gamma}\delta^2- \delta\right),
\end{align}
 where  $C_{\gamma}=4 \gamma^{-\frac{1}{2}}(C_1+\gamma^{-1}\delta) \frac{\sqrt{\gamma}+\delta}{\sqrt{\gamma}-\gamma}$.
\end{lemma}

\begin{remark}This bound is only informative when the right-hand side is positive, which necessitates  $ {\gamma<1}$   and $\tilde{\rho}_{\min} > \sqrt{\gamma}$.
\end{remark}

\begin{proof} 
Take  $\lambda=i\beta$, with $\beta\geq \beta_0=\max \{ \gamma^{\frac{1}{4}}+\sqrt{\delta},\rho(B)\}$.  
Then
\begin{align}
    m_i^{(1)}(\lambda)&=1-\sum_{k\in [m]} \frac{|X_{ik}|^2}{\beta^2+|X_{ik}|^2}\geq 1-\frac{c}{\beta^2}(1+\delta)\notag\\
    &\geq 1-\frac{\gamma(1+\delta)}{\sqrt{\gamma}+\delta}= \frac{\delta (1-\gamma) + \sqrt{\gamma} - \gamma}{\sqrt{\gamma} + \delta} \geq \frac{\sqrt{\gamma} -\gamma}{\sqrt{\gamma}+\delta}=:C_2, \label{eq:C_2}
\end{align}
where $C_2$ is lower bounded by $\frac{\sqrt{\gamma}-\gamma}{1+\sqrt{\gamma}}$. This implies $M_1(\lambda)$ is invertible.

Define $H_2(\lambda)=X^*(\lambda) M_1(\lambda)^{-1} X(\lambda)$.
  From \eqref{eq:IB-bipartite}, $\lambda\in \sigma(B)$ if and only if $\det(M_2(\lambda)-H_2(\lambda))=0.$
Recall that when $\lambda=i\beta$, $M_2(\lambda)$ is a real diagonal matrix, then $H_2(\lambda)$ 
is Hermitian. As $\beta\to\infty$,  $M_2(\lambda)-H_2(\lambda)=I+O(\beta^{-2}).$
By continuity, $\det (M_2(\lambda)-H_2(\lambda))>0$ for $\beta>\beta_0$ and $M_2(\lambda)-H_2(\lambda)$ is positive semidefinite for any $\beta\geq \beta_0$.
Since
\begin{align*}
    &\beta^2 X^*(\lambda) M_1(\lambda)^{-1} X(\lambda) + X^* M_1(\lambda)^{-1}X\\
    =&~\beta^2 X^*(\lambda) M_1(\lambda)^{-1} X(\lambda)+\lambda X^* (\lambda) M_1(\lambda)^{-1} X-\lambda X^* (\lambda) M_1(\lambda)^{-1} X+X^* M_1(\lambda)^{-1}X,
\end{align*}
by triangle inequality,
\begin{align}
  \|\beta^2 X^*(\lambda) M_1(\lambda)^{-1} X(\lambda) + X^* M_1(\lambda)^{-1}X\| 
 \leq  & \|\beta^2 X^*(\lambda) M_1(\lambda)^{-1} X(\lambda)+\lambda X^* (\lambda) M_1(\lambda)^{-1} X\| \label{eq:triangleineq1}\\
 &+\|\lambda X^* (\lambda) M_1(\lambda)^{-1} X-X^* M_1(\lambda)^{-1}X\|.\label{eq:triangleinequality}
\end{align}

For the term in  \eqref{eq:triangleineq1},
\begin{align}
    &\|\beta^2 X^*(\lambda) M_1(\lambda)^{-1} X(\lambda)+\lambda X^* (\lambda) M_1(\lambda)^{-1} X\| 
    =\|\lambda^2 X^*(\lambda) M_1(\lambda)^{-1} X(\lambda)-\lambda X^* (\lambda) M_1(\lambda)^{-1} X\| \notag  \\
   \leq &\|\lambda X(\lambda) -X\| \|\lambda X(\lambda)\| \| M_1(\lambda)^{-1}\| 
   \leq \|\lambda X(\lambda) -X\| \| M_1(\lambda)^{-1}\|( \|X\| + \| \lambda X(\lambda)-X\|).\label{eq:trianglebound1}
\end{align}

Rewriting \eqref{eq:lambdaH} with $\lambda=i\beta$, we obtain
\begin{align}
    |\lambda H_{ij}(\lambda)-H_{ij}| =\frac{|H_{ij}|^3}{\beta^2+|H_{ij}|^2}\leq \frac{\delta |H_{ij}|^2}{\beta^2}.
\end{align}
Then, applying the Gershgorin circle theorem to the row of $H$ yields,
\begin{align}\label{eq:approxlambdaX}
 \|\lambda X(\lambda)-X\|=    \|\lambda H(\lambda)-H\|\leq \frac{\delta (1+\delta)}{\beta^2}\leq \frac{1}{\sqrt{\gamma}} \delta (1+\delta)\leq  2\gamma^{-\frac{1}{2}}\delta.
\end{align}
Then with \eqref{eq:trianglebound1}  and \eqref{eq:approxlambdaX}, the  term in  \eqref{eq:triangleineq1} satisfies 
\begin{align}\label{eq:lambdaX1}
    \|\beta^2 X^*(\lambda) M_1(\lambda)^{-1} X(\lambda)+\lambda X^* (\lambda) M_1(\lambda)^{-1} X\|&\leq (2\gamma^{-\frac{1}{2}}\delta)(C_1+2\gamma^{-1}\delta)C_2^{-1}.
\end{align}
Similarly, the second term in \eqref{eq:triangleinequality} satisfies
\begin{align}\label{eq:lambdaX2}
    \|\lambda X^* (\lambda) M_1(\lambda)^{-1} X-X^* M_1(\lambda)^{-1}X\|
    \leq & \|\lambda X^*(\lambda)-X^*\| \|X\| \| M_1(\lambda)^{-1}\|
    \leq  (2\gamma^{-\frac{1}{2}}\delta)C_1C_2^{-1}. 
\end{align}

Therefore from \eqref{eq:trianglebound1}, \eqref{eq:lambdaX1}, and \eqref{eq:lambdaX2},
\begin{align*}
    \|\beta^2 X^*(\lambda) M_1(\lambda)^{-1} X(\lambda) + X^* M_1(\lambda)^{-1}X\| \leq 4\delta \gamma^{-\frac{1}{2}}(C_1+\gamma^{-1}\delta) C_2^{-1}=C_{\gamma}\delta.
\end{align*}
which implies
\begin{align}\label{eq:PSDorder1}
    X^* M_1(\lambda)^{-1}X & \succeq -\beta^2 X^*(\lambda) M_1(\lambda)^{-1} X(\lambda) -C_{\gamma}\delta=-\beta^2 H_2(\lambda)-C_{\gamma}\delta  \succeq -\beta^2 M_2(\lambda)-C_{\gamma}\delta, 
\end{align}
where  we used the condition $M_2(\lambda)-H_2(\lambda)\succeq 0$ for $\beta\geq \beta_0$. On the other hand, for any $j\in [m]$,
\begin{align*}
    \beta^2m_j^{(2)}(\lambda)&=\beta^2+\sum_{k\in [n]}\frac{\beta^2X_{kj}^2}{-\beta^2-X_{kj}^2}\leq \beta^2-\frac{\beta^2}{
    \beta^2+\delta^2}\sum_{k\in [n]}X_{kj}^2\leq \beta^2-\frac{\beta^2}{\beta^2+\delta^2}\tilde{\rho}_{\min}(1-\delta).
\end{align*}
Then
\begin{align}\label{eq:PSDorder2}
 -\beta^2 M_2(\lambda)\succeq -\left(\beta^2-\frac{\beta^2}{
    \beta^2+\delta^2}(1-\delta) \tilde{\rho}_{\min}\right).
\end{align}
Hence,  with \eqref{eq:PSDorder2} and \eqref{eq:PSDorder1},
$X^* M_1(\lambda)^{-1}X\succeq  \frac{\beta^2}{\beta^2+\delta^2}\tilde{\rho}_{\min}-\beta^2-C_{\gamma}\delta- \delta$.
From \eqref{eq:C_2}, we obtain
$
    X^* X\succeq  C_2\left( \frac{\beta^2}{
    \beta^2+\delta^2}\tilde{\rho}_{\min}-\beta^2-C_{\gamma}\delta-\tilde{\rho}_{\min}\delta\right).
$
Thus, the lower bound on $\sigma_{\min}(X) $ in \eqref{eq:XXlowerbound} holds.
\end{proof}

\section{Probabilistic upper bound on  $\rho(B)$}\label{sec:bound_rho_B}

In this section, we provide a probabilistic bound on the spectral radius $\rho(B)$ for a random matrix $H$. This involves sharp estimates of the trace of high powers of $B$.  We give our main result below.

\begin{theorem}\label{thm:rhoB}
Let $H=\begin{bmatrix}
0 & X\\
X^{*} & 0
\end{bmatrix}$ be an $(n+m)\times (n+m)$ random  Hermitian matrix with associated non-backtracking matrix $B$, where $X$ is an $n\times m$ random matrix with centered independent entries.   {Under Assumption~\ref{assumption_3}},
for $\varepsilon> 0$, there exist universal constants $C,c_1>0$ such that
\[
   \mathbb P (\rho(B)\geq \gamma^{\frac{1}{4}}(1+\varepsilon))\leq C\gamma^{-\frac{5}{6}}N^{3-c_1q\log(1+\varepsilon)}  . \]
\end{theorem}

\begin{remark}
The bound on $\rho(B)$ in Theorem \ref{thm:rhoB} is an inhomogenous analog to  \cite[Theorem 3]{brito2022spectral} for the non-backtracking operator of random biregular bipartite graphs, and for a broader range of the (sparsity) parameter $q$.  Compared to \cite[Theorem 2.5]{benaych2020spectral}, we improve a factor $\gamma^{1/4}$ in the bound of $\rho(B)$ when  $H$ has a  bipartite block structure.
\end{remark}

The key estimate to prove Theorem \ref{thm:rhoB} is the following trace bound on a high power of $B$.
\begin{lemma}\label{lem:trace} Suppose $X$ satisfies Assumption \ref{assumption_3}. There exist  universal constants  $c_0,C$ such that for any $\delta \in (0,\frac{1}{3})$, and any  {odd} positive integer  $l$  satisfying
\begin{align} \label{eq:assumptionl}
    l\leq c_0\min \left\{ \delta q\log N, \frac{N^{\frac{1}{3}-\delta}}{\kappa^{\frac{1}{3}}\gamma^{\frac{1}{6}}q^2}\right\},
\end{align}
we have
\begin{align}\label{eq:NBtrace}
    \mathbb E\tr[B^l (B^l)^*]\leq Cl^4q^2mn \gamma^{(l-1)/2}.
\end{align}
\end{lemma}

Assuming Lemma \ref{lem:trace}, we can now prove Theorem \ref{thm:rhoB}.  
\begin{proof}[Proof of Theorem \ref{thm:rhoB}]
Choose  $0<\delta<\frac{1}{30}$ in \eqref{eq:assumptionl}, we can take $l=\lceil\frac{1}{2}c_1q\log N \rceil$ for some sufficiently small constant $c_1>0$ then \eqref{eq:assumptionl} is satisfied. By Markov's inequality,
\begin{align*}
   \mathbb P (\rho(B)\geq \gamma^{\frac{1}{4}}(1+\varepsilon))&\leq  \frac{ \mathbb E\tr[B^l (B^l)^*]}{\gamma^{l/2}(1+\varepsilon)^{2l}}\leq Cl^4q^2mn\gamma^{-\frac{1}{2}}\left(\frac{1}{1+\varepsilon}\right)^{2l}\\
   &\leq CN^2q^6(\log N)^4\gamma^{-\frac{1}{2}}(1+\varepsilon)^{-c_1q\log (1+\varepsilon)}\leq C\gamma^{-\frac{5}{6}} N^{3-c_1q\log (1+\varepsilon)}
\end{align*}
for some absolute constant $C$. In the last inequality, we used the upper bound on $q$ in the assumption of Theorem \ref{thm:rhoB}. This finishes the proof.
 \end{proof}

\section{Proof of Lemma \ref{lem:trace}}\label{sec:trace}
We now provide the proof  Lemma \ref{lem:trace}. The proof is adapted from \cite[Proposition 5.1]{benaych2020spectral} to the non-backtracking operator $B$ of a Hermitian random matrix $H$ with a bipartite block structure. The improvement compared to \cite{benaych2020spectral} is the factor $\gamma^{(l-1)/2}$ in  \eqref{eq:NBtrace}, by doing path counting on a complete bipartite graph instead of a complete graph.

Let $V_1, V_2$ be the left and right vertex sets on the complete graph $K_{n,m}$, and \[\vE= \{ (u,v): u\in V_1, v\in V_2\}\cup \{ (u,v): u\in V_2, v\in V_1\}\] be the set of all oriented edges in $K_{n,m}$. For any $e\in \vE, f\in \vE$, from the definition of $B$,
\begin{align*}
    (B^l)_{ef}&=\sum_{a_1,\dots, a_{l-1}\in \vec E} B_{ea_1}B_{a_1a_2}\cdots B_{a_lf}=\sum_{\xi} H_{\xi_0\xi_1}H_{\xi_1\xi_2}\cdots H_{\xi_{l-1}\xi_l},
\end{align*}
where the sum on the right hand side runs over $\xi=(\xi_{-1},\xi_0,\dots,\xi_{l})$ as a path of length $l+1$ in $K_{n,m}$ with $(\xi_{-1},\xi_0)=e, (\xi_{l-1},\xi_l)=f$, and $\xi_{i-1}\not=\xi_{i+1}$ for $0\leq i\leq l-1$. Therefore, we have
\begin{align}
    \tr[B^l (B^l)^*]&=\sum_{e,f\in \vE} |(B^{l})_{ef}|^2\label{eq:sumterm}\\
    &=\sum_{\xi^1,\xi^2}H_{\xi_0^1\xi_1^1}H_{\xi_1^1\xi_2^1}\cdots H_{\xi_{l-1}^1\xi_l^1}    H_{\xi_l^2\xi_{l-1}^2}\cdots   H_{\xi_{1}^2\xi_0^2},\label{eq:sumterm2}
\end{align}
where the sum runs over paths $\xi^1=(\xi_{-1}^1,\dots,\xi_{l}^1),\xi^2=(\xi_{-1}^2,\dots,\xi_{l}^2)$ of length $l+1$ such that $(\xi_{-1}^1,\xi_0^1)=(\xi_{-1}^2,\xi_{0}^2)$, $(\xi_{l-1}^1,\xi_l^1)=(\xi_{l-1}^2,\xi_l^2)$ and $\xi_{i-1}^{j}\not=\xi_{i+1}^j$ for $j=1,2$ and $0\leq i\leq l-1$. 
From \eqref{eq:sumterm}, for fixed $\xi_{-1}^1, \xi_{-1}^2$, the sum in \eqref{eq:sumterm2} is nonnegative and does not depend on $\xi_{-1}^1,\xi_{-1}^2$. Therefore, we can bound it by
\begin{align}\label{eq:xi0}
     \tr[B^l (B^l)^*]\leq & m\sum_{\substack{\xi^1,\xi^2:\\ \xi_0^1,\xi_0^2\in V_1}}H_{\xi_0^1\xi_1^1}H_{\xi_1^1\xi_2^1}\cdots H_{\xi_{l-1}^1\xi_l^1}   H_{\xi_l^2\xi_{l-1}^2}\cdots  H_{\xi_{1}^2\xi_0^2}\\
     &+n\sum_{\substack{\xi^1,\xi^2:\\ \xi_0^1,\xi_0^2\in V_2}}H_{\xi_0^1\xi_1^1}H_{\xi_1^1\xi_2^1}\cdots H_{\xi_{l-1}^1\xi_l^1}   H_{\xi_l^2\xi_{l-1}^2}\cdots  H_{\xi_{1}^2\xi_0^2}, \notag
\end{align}
where the sum is over $\xi^1=(\xi_0^1,\dots,\xi_l^1), \xi^2=(\xi_0^2,\dots,\xi_l^2)$ such that $(\xi_0^1,\xi_{l-1}^1,\xi_l^1)=(\xi_0^2,\xi_{l-1}^2,\xi_l^2)$ and $\xi_{i-1}^j\not=\xi_{i+1}^j$ for $j=1,2$ and $1\leq i\leq l-1$. 
For $j=1,2$, define
    \begin{align*}
        \tilde{\mathcal C}_j=\{ \xi=(\xi_0,\dots,\xi_{2l}): \xi_0=\xi_{2l}\in V_j,\xi_{l-1}=\xi_{l+1}, \xi_{i-1}\not=\xi_{i+1}, \forall  1\leq i\leq 2l-1 ,i\not=l\}.
    \end{align*}
We can combine each pair of $\xi^1, \xi^2$  in \eqref{eq:xi0} into a path of length $2l$ and simplify the bound as
\begin{align} \notag 
 \tr[B^l (B^l)^*]\leq & m\sum_{\xi\in\tilde{\mathcal C}_1} H_{\xi_0\xi_1}H_{\xi_1\xi_2}\cdots H_{\xi_{2l-1}\xi_{2l}}   +n\sum_{\xi\in  \tilde{\mathcal C}_2}H_{\xi_0\xi_1}H_{\xi_1\xi_2}\cdots H_{\xi_{2l-1}\xi_{2l}}. 
    \end{align}
    Since the entries of $H$ are independent up to symmetry with mean zero,  taking the expectation yields
\begin{align} \label{eq:expectedtrB}
   \E\tr[B^l (B^l)^*]\leq & m\sum_{\xi\in{\mathcal C}_1} \E H_{\xi_0\xi_1}H_{\xi_1\xi_2}\cdots H_{\xi_{2l-1}\xi_{2l}} +n\sum_{\xi\in{\mathcal C}_2}\E H_{\xi_0\xi_1}H_{\xi_1\xi_2}\cdots H_{\xi_{2l-1}\xi_{2l}},
\end{align}
where $\mathcal C_i$ is a subset of $\tilde{\mathcal C}_i$ for $i=1,2$ such that each edge in the graph spanned by $\xi$ is visited at least twice by the path defined by $\xi$.

We will combine paths according to their equivalence classes defined as follows.
\begin{definition}[Equivalence class]
For each $\xi\in \mathcal C_1\cup \mathcal C_2$, we define a graph \[G_{\xi}=(V_1(\xi), V_2(\xi), E(\xi))\] spanned by $\xi$. Let $g(\xi)=|E(\xi)|-|V(\xi)|+1$ be the genus of  $G_{\xi}$ and $e(\xi)=|E(\xi)|$. We say two paths $\xi_1,\xi_2$ are equivalent if they are the same up to a relabelling of vertices (with the restriction that after relabelling, vertices in $V_i$  stay in $V_i$ for $i=1,2$).
Define $[\xi]$ the equivalence class of a path $\xi$.
\end{definition}

\begin{lemma}\label{lem:C1C2trB}
For any $\overline{\xi}\in \mathcal C_1$ with $|V_1(\overline\xi)|=s_1(\overline\xi), |V_2(\overline\xi)|=s_2(\overline\xi)$, we have
\begin{align}\label{eq:V1}
    \E \sum_{\xi\in [\overline{\xi}]} H_{\xi_0\xi_1} H_{\xi_1\xi_2}\cdots H_{\xi_{2l-1}\xi_{2l}}\leq n(\kappa/N)^{g(\overline \xi)} q^{2e(\overline \xi)-2l}\rho_{\max}^{s_1(\overline \xi)}\tilde{\rho}_{\max}^{s_2(\overline \xi)-1}.
\end{align}
And for any $\overline{\xi}\in \mathcal C_2$,
\begin{align}\label{eq:V2}
    \E \sum_{\xi\in [\overline{\xi}]} H_{\xi_0\xi_1} H_{\xi_1\xi_2}\cdots H_{\xi_{2l-1}\xi_{2l}}\leq m (\kappa/N)^{g(\overline \xi)} q^{2e(\overline \xi)-2l}\rho_{\max}^{s_1(\overline \xi)-1}\tilde{\rho}_{\max}^{s_2(\overline \xi)}.
\end{align}
\end{lemma}
\begin{proof}
We consider a breadth-first search ordering of vertices in $\overline\xi$. Let $a=|E(\overline \xi)|$ and $s=s_1+s_2$, where the drop the dependence on $\overline \xi$ in $s_1(\overline\xi)$ and $s_2(\overline\xi)$ for convenience. We order the edges in $G(\overline \xi)$ such that the edges $e_1,\dots, e_{s-1}$ form a spanning tree of $G(\overline{\xi})$, and let $m_t$ be the number of times that $e_t$ appear in $\overline{\xi}$ for $1\leq t\leq s-1$. Let $\mathcal I_{s_1,s_2}$ be the set of injection maps from $[s_1]\times [s_2]$ to $[n]\times [m]$. Then for any $\overline{\xi} \in \mathcal C_1$,
\begin{align*}
     \E \sum_{\xi\in [\overline{\xi}]} H_{\xi_0\xi_1} H_{\xi_1\xi_2}\cdots H_{\xi_{2l-1}\xi_{2l}}&=\sum_{\tau\in \mathcal I_{s_1,s_2}}\E H_{\tau(\overline{\xi}_0)\tau(\overline{\xi}_1)}\cdots H_{\tau(\overline{\xi}_{2l-1})\tau(\overline{\xi}_{2l})}
     \leq \sum_{\tau\in \mathcal I_{s_1,s_2}}\prod_{t=1}^a \E | H_{\tau(e_t)}^{m_t}|.
\end{align*}
From Assumption~\ref{assumption_3}, we can  use the estimates
\begin{align*}
    \max_{j\in V_2}\sum_{i\in V_1}\E |H_{ij}|^k\leq \frac{\tilde{\rho}_{\max}}{q^{k-2}}, \quad 
     \max_{i\in V_1}\sum_{j\in V_2}\E |H_{ij}|^k\leq \frac{\rho_{\max}}{  q^{k-2}}
\end{align*}
for contribution from edges $e_1,\dots, e_{s-1}$, and the estimate $ 
    \max_{ij} \E |H_{ij}|^k \leq \frac{\kappa}{Nq^{k-2}}$
for $e_{s},\dots, e_{a}$ to obtain 
\begin{align}
    \E \sum_{\xi\in [\overline{\xi}]} H_{\xi_0\xi_1} H_{\xi_1\xi_2}\cdots H_{\xi_{2l-1}\xi_{2l}}&\leq \prod_{t=s}^a \frac{\kappa}{Nq^{m_t-2}}\sum_{\tau\in \mathcal I_{s_1,s_2}}\prod_{t=1}^{s-1}\E |H_{\tau(e_t)}^{m_t}| \notag \\
    &\leq \left(\prod_{t=s}^a \frac{\kappa}{Nq^{m_t-2}} \right)n \rho_{\max}^{s_1}\tilde{\rho}_{\max}^{s_2-1} \frac{1}{q^{m_1-2}}\cdots \frac{1}{q^{m_{s-1}-2}}\label{eq:factorn}\\
    &=n q^{2a-\sum_{t=1}^a m_t} \left(\frac{\kappa}{N}\right)^{a-s+1}\rho_{\max}^{s_1}\tilde{\rho}_{\max}^{s_2-1} =n(\kappa/N)^g q^{2a-2l}\rho_{\max}^{s_1}\tilde{\rho}_{\max}^{s_2-1}, \notag
\end{align}
where $g=a-v+1$ is the genus of $G(\overline{\xi})$. The factor $n$ in \eqref{eq:factorn} comes from the processing of pruning trees, and the root of the spanning tree has at most $n$ choices of labeling since $\overline{\xi} \in \mathcal C_1$. Then \eqref{eq:V1} holds. \eqref{eq:V2} follows similarly.
\end{proof}

We can further simplify the upper bound on  $\E\tr[B^l (B^l)^*]$ by counting the contributions from normal graphs defined below.
\begin{definition}[Normal graph]
Denote a path as $w=w_0w_{01}w_1w_{12}\cdots w_{l-1,l}w_l$ in a multigraph $G$, where $w_0,\dots,w_l$ are the vertices and $w_{i,i+1}, 0\leq i\leq l-1$ are the edges  visited by the path $w$. We say $w$ in $G$ is \textit{normal} if 
\begin{itemize}
    \item $V(G)=[s]$ where $s=|V(G)|$;
    \item  the vertices in $V(G)$ are visited in increasing order by $w$.
\end{itemize}
\end{definition}

Each equivalent class of $\mathcal C_i$ has a unique representative $\xi$ such that $\xi$ is normal in $G_{\xi}$.
For $i=1,2$, denote $\mathcal C_{0,i}:= \{ \xi\in \mathcal C_i: \xi \text{ is normal in } G_{\xi}\}$.
From \eqref{eq:expectedtrB} and Lemma \ref{lem:C1C2trB}, we obtain
\begin{align}
    \E\tr[B^l (B^l)^*]&\leq mn\sum_{\xi\in \mathcal C_{0,1}}(\kappa/N)^{g(\xi)} q^{2e(\xi)-2l}\rho_{\max}^{s_1(\xi)}\tilde{\rho}_{\max}^{s_2(\xi)-1} \notag\\
    &+m n\sum_{\xi\in \mathcal C_{0,2}}(\kappa/N)^{g(\xi)} q^{2e(\xi)-2l}\rho_{\max}^{s_1(\xi)-1}\tilde{\rho}_{\max}^{s_2(\xi)} =S_1+S_2.\label{eq:S1S2}
\end{align}

From now on, we only treat $S_1$, and $S_2$ can be estimated similarly.
We now introduce a parametrization of $\mathcal C_{0,1}$ following \cite{benaych2020spectral}.

\begin{definition}[Equivalence class]
Let $G$ be a graph and $\mathcal V\subset V(G)$. Define 
\[ \mathcal I_{\mathcal V}(G)=\{ v\in V(G)\setminus \mathcal V: \deg(v)=2\}.\]

Let $\Sigma_{\mathcal V}(G)$ be the set of paths $w=w_0\cdots w_l$ in $G$ such that $w_1,\dots, w_{l-1}$ are pairwise distinct and belong to $\mathcal I_{\mathcal V}(G)$ and $w_0,w_l\not\in \mathcal I_{\mathcal V}(G)$. We define an equivalence relation on $\Sigma_{\mathcal V}(G)$ such that the path $w_0\dots w_l$ and the reverse ordered path $w_{l}\cdots w_1$ are equivalent. Denote $\Sigma'_{\mathcal V}(G)=\{[w]: w\in \Sigma_{\mathcal V}(G) \}$ the set of equivalence classes.
\end{definition}

We can construct a multigraph $\hat{G}_{\xi}$ from $G_{\xi}$ by replacing every $[w]\in \Sigma'_{\xi_0,\xi_l}(G)$ with an edge in $E(\hat{G}_{\xi})$ as follows.

\begin{definition}[Multigraph $\hat{G}_{\xi}$ from $G_{\xi}$]\label{def:multigraphconstruct}
Let $\xi\in \mathcal C_{0,1}$. Define $ V(\hat G_{\xi})=V(G_{\xi})\setminus \mathcal I_{\xi_0,\xi_{\ell}}(G_{\xi})$ and $E(\hat G_{\xi})=\Sigma'_{\xi_0,\xi_l}(G)$. The endpoints of $[w]$ in $E(\hat{G}_{\xi})$ are labeled $w_0, w_l$. Assign each edge $[w]\in E(\hat{G}_{\xi})$ the weight $\hat{k}_{w}$, which is the length of the path $w$.
\end{definition}

From \cite{benaych2020spectral}, any $\xi\in\mathcal C_{0,1}$ as a closed path $\xi_0\xi_1\cdots \xi_{2l}$ in $G_{\xi}$ gives rise to a closed path $\hat{\xi}=\hat{\xi_0}\hat{\xi}_{01}\hat{\xi}_{1}\hat{\xi}_{12}\cdots \hat{\xi}_{r-1}\hat{\xi}_r$ on the multigraph $\hat{G}_{\xi}$. Now for any $\xi \in \mathcal C_{0,1}$, we have constructed a triple $(\hat{G}_\xi, \hat{\xi},\hat{k})$, where $\hat{G}_{\xi}$ is a multigraph, $\hat{\xi}$ is a closed path in $\hat{G}_{\xi}$. 
Set $\tau$ to be the unique increasing bijection from $V(\hat{G_{\xi}})$ to $\{1,\dots, |V(\hat{G_{\xi}})|\}$. Denote by $(U,\zeta, k):=(U(\xi),\zeta(\xi), k(\xi))$ the triple obtained from the triple $(\hat{G}_\xi, \hat{\xi},\hat{k})$ by relabelling the vertices using $\tau$. By Definition, $\zeta_0=\tau(\xi_0)=1$. We set $\nu=\nu(\xi)=\tau(\xi_{l})$.
Altogether, the construction above gives a map $\xi\mapsto (U,\zeta,k)$. The following lemma collects some properties of this map.

\begin{lemma}[Lemma 5.3 in \cite{benaych2020spectral}]\label{lem:propertymapping}
The map $\xi \mapsto (U,\zeta,k)$ satisfies the following properties:
\begin{enumerate}
    \item The map $\xi \mapsto (U,\zeta,k)$ is an injection on $\mathcal C_{0,1}$.
    \item $g(U)=g(G_{\xi})$.
    \item $\zeta$ is a closed path in the multigraph $U$ and it is normal in $U$.
    \item Every vertex of $V(U)\setminus \{1,\nu\}$ has degree at least three. The vertices $1$ and $\nu$ have degrees at least one.
    \item $|E(G_{\xi})|=\sum_{e\in E(U)} k_e$.
    \item $m_e(\zeta)\geq 2$ for all $e\in E(U)$ and $2l=\sum_{e\in E(U)} m_e(\zeta)k_e$.
\end{enumerate}
\end{lemma}

For a given $\xi\in \mathcal C_{0,1}$ with given $s_1(\xi), s_2(\xi)$, we know \begin{align}\label{eq:trivialboundparameter}
    s_1+s_2=|V(G_{\xi})|=e(\xi)-g(U)+1 \quad \text{ and } |s_1-s_2|\leq g(U),
\end{align} where we use $g(\xi)=g(U)$ from Lemma \ref{lem:propertymapping}.  The second claim is due to the fact that since $\ell$ is odd, any imbalance between $s_1,s_2$ is from creating a new cycle in $\xi$. Moreover,
\begin{align}
 e(\xi)&=\sum_{e\in E(U)}k_{e},\quad 
  2l=\sum_{e\in E(U)}m_e(\zeta)k_e.
\end{align}
Since $k_e\geq 1, m_e(\zeta)\geq 2$, we obtain
\begin{align}\label{eq:boundq}
   2e(\xi)-2l\leq 2|E(U)|-|\zeta|,
\end{align}
where $| \zeta |$ is the length of the closed path $\zeta$. 
Equipped with the construction of the triple $(U,\zeta, k)$, we continue to estimate $S_1$ in \eqref{eq:S1S2}.  With \eqref{eq:trivialboundparameter},
\begin{align}
 S_1=&   mn\sum_{\xi\in \mathcal C_{0,1}}  (\kappa/N)^{g(\xi)}q^{2e(\xi)-2l}\rho_{\max}^{s_1}\tilde{\rho}_{\max}^{s_2-1} \notag \\
 \leq & mn\sum_{\xi\in \mathcal C_{0,1}} (\kappa/N)^{g(\xi)} q^{2e(\xi)-2l} (\rho_{\max}\tilde{\rho}_{\max})^{(s_1+s_2-1)/2}  \max\{\rho_{\max}/\tilde{\rho}_{\max},\tilde{\rho}_{\max}/\rho_{\max}\}^{|s_1-s_2+1|/2} \notag  \\
 \leq &  mn\sum_{\xi\in \mathcal C_{0,1}} (\kappa/N)^{g(\xi)} q^{2e(\xi)-2l} \gamma^{(e(\xi)-g(\xi))/2}\left(\frac{1}{c}\right)^{(g(\xi)+1)/2} \notag \\
 =& mn\gamma^{(l-1)/2}\sum_{\xi\in \mathcal C_{0,1}} (\kappa/N)^{g(\xi)} q^{2e(\xi)-2l} \gamma^{(e(\xi)-l)/2-g(\xi)}.\label{eq:mnSum}
 \end{align}
 From (1) in Lemma \ref{lem:propertymapping}, we can upper bound \eqref{eq:mnSum} by summing over $(U,\zeta,k)$ instead of $\xi$. Then with \eqref{eq:boundq}, and the assumption $q\geq \gamma^{-\frac{1}{4}}$, we find
 \begin{align*}
     S_1\leq mn\gamma^{(l-1)/2}\sum_{(U,\zeta,k)}\left(\frac{\kappa}{Nc}\right)^{g(U)}\left(q\gamma^{\frac{1}{4}}\right)^{2|E(U)|-|\zeta|},
 \end{align*}
 where $(U,\zeta,k)$ is obtained from all $\zeta \in \mathcal C_{0,1}$.
Since
$\sum_{e\in E(U)} k_e m_e(\zeta)=2l$ and $m_e(\zeta)\geq 2$, for a given $(U,\zeta)$, the number of choices for $k=(k_e)_{e\in E(U)}$ can be bounded by the number of $k$ such that $\sum_{e\in E(U)}k_e=2l$. For fixed $(U,\zeta)$, the number of choices for such  $k$ is bounded by 
\[\binom{2l-1}{|E(U)|-1}\leq \left( \frac{6l}{|E(U)|}\right)^{|E(U)|}.\]
Therefore
\begin{align}\label{eq:S1Uzeta}
    S_1\leq mn\gamma^{(l-1)/2}\sum_{(U,\zeta)}\left(\frac{6l}{|E(U)|} \right)^{|E(U)|}\left(\frac{\kappa}{Nc}\right)^{g(U)} \left(q\gamma^{\frac{1}{4}}\right)^{2|E(U)|-|\zeta|} .
\end{align}
From  \cite[Lemma 5.8]{benaych2020spectral},  we have
\begin{align}\label{eq:implies}
 |E(U)|\leq 3g(U)+1, \quad 
    |V(U)|\leq 2g(U)+2.
\end{align}
Also, $|E(U)|\geq g(U)\vee 1 $ and $g(U)\leq l$ by definition.
Therefore, \eqref{eq:S1Uzeta} can be further bounded by 
\begin{align}\label{eq:324}
    S_1\leq 6l q^2mn\gamma^{l/2}\sum_{(U,\zeta)}\left(\frac{12l}{g(U)+1} \right)^{3g(U)}\left(\frac{\kappa}{Nc}\right)^{g(U)} \left(q\gamma^{\frac{1}{4}}\right)^{6g(U)-|\zeta|} .
\end{align}
\eqref{eq:implies} implies the number of pairs $(U,\zeta)$ such that $U$ has genus $g$ and $\zeta$ has length $t$ is bounded by
\begin{align}\label{eq:boundg}
  (3g+1)^t(2g+2)^{3g+1}.
\end{align}
With \eqref{eq:boundg} and  \eqref{eq:324}, we find for some absolute constant $C>0$,
\begin{align}
    S_1 &\leq 6lq^2mn\gamma^{(l-1)/2}\sum_{g=0}^{l}\sum_{t=1}^{2l}(3g+1)^t(2g+2)^{3g+1}\left(\frac{24l}{2g+2} \right)^{3g}\left(\frac{\kappa}{Nc}\right)^{g} \left(q\gamma^{\frac{1}{4}}\right)^{6g-t}  \notag\\
    &\leq Clq^2mn\gamma^{(l-1)/2}\sum_{t=1}^{2l}(q\gamma^{\frac{1}{4}})^{-t}+Cl^2q^2mn\gamma^{(l-1)/2}\sum_{g=1}^l \left(\frac{Cl^3\kappa \gamma^{\frac{1}{2}}q^6}{N} \right)^{g}\sum_{t=1}^{2l}  \left( \frac{4g}{q\gamma^{\frac{1}{4}}}\right)^t. \label{eq:S11st}
\end{align}
Since $q\geq \gamma^{-\frac{1}{4}}$, the first term is bounded by $Cl^2q^2mn\gamma^{(l-1)/2}$. For the second term, using \[\sum_{m=1}^{2l} x^m\leq 2l(1+x^{2l}),\] it is bounded by 
\begin{align}\label{eq:S12ndterm}
Cl^3q^2mn\gamma^{(l-1)/2}\sum_{g=1}^l \left(\frac{Cl^3\kappa \gamma^{\frac{1}{2}}q^6}{N} \right)^{g}+ Cl^3q^2mn\gamma^{(l-1)/2}\sum_{g=1}^l
\left(\frac{Cl^3\kappa \gamma^{\frac{1}{2}}q^6}{N} \right)^{g}\left( \frac{4g}{q\gamma^{\frac{1}{4}}}\right)^{2l}.
\end{align}
From the assumption \eqref{eq:assumptionl}, we obtain 
\begin{align*}
   N &\geq \left(\frac{\kappa \gamma^{\frac{1}{2}}q^6l^3}{c_0^3} \right)^{1/(1-3\delta)}, \quad N^{3\delta}\leq \frac{c_0^3N}{\kappa \gamma^{\frac{1}{2}}q^6l^3}, \\
   l &\leq c_0\delta q\log N=\frac{c_0q}{3}\log (N^{3\delta})\leq \frac{c_0q}{3} \log \left(\frac{c_0^3N}{\kappa \gamma^{\frac{1}{2}}q^6l^3} \right).
\end{align*}
By choosing $c_0$ small enough, we have the following inequalities:
\begin{align}\label{eq:assumNl}
   N\geq 2Cl^3\kappa \gamma^{\frac{1}{2}}q^6, \quad l\leq \frac{1}{8}q\gamma^{\frac{1}{4}}\log \left(\frac{N}{Cl^3\kappa \gamma^{\frac{1}{2}}q^6}\right).
\end{align}
The first term in \eqref{eq:S12ndterm} is bounded by $Cl^3q^2mn\gamma^{(l-1)/2}$. The second term can be written as
\[ Cl^3q^2mn\gamma^{(l-1)/2}\sum_{g=1}^l
\exp\left( -g\log \left(\frac{N}{Cl^3\kappa \gamma^{\frac{1}{2}}q^6}\right)+2l \log \frac{4g}{q\gamma^{\frac{1}{4}}}\right).
\]
The argument in the exponential function is maximized at \[g=\frac{2l}{\log \left(\frac{N}{Cl^3\kappa \gamma^{\frac{1}{2}}q^6}\right)}.\]
From \eqref{eq:assumNl}, this maximizer is reached for $g\leq \frac{1}{4}q\gamma^{\frac{1}{4}}$, and the second term in \eqref{eq:S12ndterm} can be bounded by $Cl^4q^2mn \gamma^{(l-1)/2}$. 

Therefore, with the bound on \eqref{eq:S11st} and \eqref{eq:S12ndterm}, we obtain 
$S_1\leq Cl^4q^2mn \gamma^{(l-1)/2}$.
Repeating the same argument for $S_2$ yields the same upper bound.  This completes the proof of Lemma~\ref{lem:trace}.

\section{Probabilistic bounds on the largest  singular  value}\label{sec:prob_largest}

In this section, we first prove Theorem \ref{thm:lambdamax} for general rectangular random matrices, then specify the model parameters to prove Theorem \ref{thm:main_sparse_max} for sparse rectangular random matrices.
The proof is based on the deterministic spectral relations between $B$  and $H$ in Section \ref{sec:Ihara}, and the bound on $\rho(B)$ in Section \ref{sec:bound_rho_B}.

\begin{proof}
[Proof of Theorem \ref{thm:lambdamax}]
From the assumption of Theorem \ref{thm:lambdamax}, we have  \[|X_{ij}|^2\leq \frac{1}{q^2}, \quad \max_i\sum_{j} \mathbb E |X_{ij}|^2\leq  \rho_{\max}, \quad \text{and} \quad \max_i\sum_{j}  \mathbb E |X_{ij}|^4\leq \frac{\rho_{\max}}{q^2}.\] 
Since $\sum_{j} |X_{ij}|^2$ is a sum of independent bounded random variables, applying  Bennett's inequality \cite[Theorem 2.9]{boucheron2013concentration}, we obtain, for $\delta>0$,  
\begin{align}\label{eq:rhomaxbound}
    \mathbb P \left( \sum_{j} |X_{ij}|^2\geq \rho_{\max}(1+\delta )\right)\leq \exp\left(-q^2\rho_{\max} h\left(\delta \right)\right),
\end{align}
where $h(\delta)=(1+\delta)\log(1+\delta)-\delta$. Similarly,
\begin{align}\label{eq:tilderhomaxbound}
    \mathbb P \left( \sum_{i} |X_{ij}|^2\geq \tilde\rho_{\max}(1+\delta )\right)\leq \exp\left(-q^2\tilde\rho_{\max} h\left(\delta \right)\right).
\end{align}
Then, by taking a union bound, 
\begin{align*}
    \mathbb P\left( \|X\|_{2,\infty}\geq \sqrt{\rho_{\max}}(1+\delta )\right)&\leq  \mathbb P\left( \|X\|_{2,\infty}\geq \sqrt{\rho_{\max}}\cdot \sqrt{1+\delta \vee \delta^2}\right)\\
    &= \mathbb P\left( \|X\|_{2,\infty}^2\geq  \rho_{\max}(1+\delta \vee \delta^2)\right)\leq n\exp(-q^2 \rho_{\max}h( \delta \vee \delta^2)),
\end{align*}
and 
$ \mathbb P\left( \|X^*\|_{2,\infty}\geq \sqrt{\tilde\rho_{\max}}(1+\delta )\right)\leq m\exp(-q^2 \tilde\rho_{\max}h(\delta \vee \delta ^2)).$
From Theorem \ref{thm:rhoB},
\[\mathbb P (\rho(B)\geq \gamma^{\frac{1}{4}}(1+\delta ))\leq C\gamma^{-\frac{5}{6}}N^{3-c_1q\log(1+\delta)}.\]
Therefore, conditioned on a high probability event, we have \begin{align}\label{eq:high_prob_event}
\|H\|_{1,\infty}\leq q^{-1},\quad  \|H\|_{2,\infty}\leq 1+\delta, \quad \mathrm{and} \quad \rho(B)\leq \gamma^{\frac{1}{4}}(1+\delta).
\end{align}
Now we apply the deterministic upper bound on $\sigma_{\max}(X)$ given in  \eqref{eq:CorK} conditioned on \eqref{eq:high_prob_event}.  If $\rho(B)\leq \gamma^{\frac{1}{4}} \|H\|_{2,\infty}$, then 
\begin{align*}
\sigma^2_{\max}(X)&\leq (1+\delta)^2 (\sqrt{\gamma}+1)^2+48\gamma^{-\frac{5}{4}}(1+\delta) q^{-1}+36\gamma^{-2}q^{-2}\leq  \left( (\sqrt{\gamma}+1)+C_1(\delta +\gamma^{-\frac{5}{4}}q^{-1})\right)^2
\end{align*}
for some universal constant $C_1>0$.
If instead $ \gamma^{\frac{1}{4}} \|H\|_{2,\infty}<\rho(B)\leq \gamma^{\frac{1}{4}}(1+\delta)$,
we find 
\begin{align*}
  \sigma^2_{\max}(X)&\leq   \|H\|_{2,\infty}^2 f\left( \frac{\gamma^{\frac{1}{4}}(1+\delta)}{\|H\|_{2,\infty}}\right)+12\gamma^{-\frac{5}{4}}g\left( \frac{\gamma^{\frac{1}{4}}(1+\delta)}{\|H\|_{2,\infty}}\right) \|H\|_{2,\infty} q^{-1} +36\gamma^{-2}q^{-2}\\
  &\leq \left( (\sqrt{\gamma}+1)+C_1(\delta +\gamma^{-\frac{3}{2}}q^{-1})\right)^2.
\end{align*}
Combining both cases, for any $\delta>0$,  with probability at least 
\begin{align}\label{eq:prob_tail}
1-C\gamma^{-\frac{5}{6}} N^{3-c_1q\log (1+\delta)}-2Ne^{-\gamma q^2h(\delta\vee \delta^2)},\\
    \sigma_{\max}(X)\leq   \sqrt{\gamma}+1+C_1(\delta +\gamma^{-\frac{3}{2}}q^{-1})  .\label{eq:uppp}
\end{align}

Next, we simplify the probability tail bound \eqref{eq:prob_tail} by picking a specific $\delta$. Let  $K\geq 1$ and 
\begin{align}\label{eq:def_eta}
\delta=\frac{K\eta}{\sqrt{1\vee \log \eta}} \quad \text{ where } \quad \eta=\frac{\sqrt{\log N}}{q}.
\end{align}
From the condition that $q\geq \gamma^{-1/4}\geq 1$, by considering $\eta\geq e, \eta<e$ separately, we have for $N\geq 2$, there exists a constant $c_2=\frac{2}{e}$ such that
\begin{align}\label{eq:qdelta}
    q\delta =\frac{K\sqrt{\log N}}{\sqrt{1\vee \log \eta}} =\frac{K\eta q}{\sqrt{1\vee \log \eta}}\geq c_2K.
\end{align}
Therefore \begin{align}\label{eq:upper_q}
    q^{-1}\leq \frac{\delta}{c_2K}.
\end{align}
Moreover,  from \eqref{eq:qdelta}, using the fact that $\frac{\log(1+x)}{x}$ is decreasing on $(0,\infty)$ and $q\geq \gamma^{-1/4}\geq 1$, we obtain
\begin{align}\label{eq:lower_a0}
    q\log(1+\delta)\geq c_2K\cdot \frac{\log(1+c_2K/q)}{c_2K/q}\geq \log(1+c_2K).
\end{align}
Now we give a lower bound on 
  $a_1:=\gamma q^2h(\delta\vee \delta^2)$.
If $\eta\leq e$, using the fact that $h(x)\geq c(x^2\wedge x)$ for all $x\geq 0$ and some universal constant $c>0$, we find 
\begin{align}\label{eq:bound_a11}
    a_1\geq c\gamma q^2 \delta^2= c\gamma K^2\log N.
\end{align}
When $\eta \geq e$, using the inequality $h(x)\geq c(x^2\wedge x)(1\vee \log x)$ for all $x\geq 0$, we obtain for $\delta\geq e$, \[h(\delta \vee \delta^2)\geq 2c\delta^2\log(\delta).\]
Since from  \eqref{eq:def_eta}$,\log(\delta)\geq c'\log (\eta)$ for some absolute constant $c'>0$,
we obtain 
\begin{align}\label{eq:bound_a12}
    a_1\geq 2c\gamma q^2 \delta^2 \log(\delta)\geq 2cc'\gamma q^2\delta^2 \log(\eta)\geq 2cc'\gamma K^2\log N.
\end{align}
From \eqref{eq:bound_a11} and \eqref{eq:bound_a12}, we conclude 
\begin{align}\label{eq:bound_a13}
    a_1\geq c\gamma K^2\log N
\end{align}
for an absolute constant $c>0$.
With \eqref{eq:upper_q}, \eqref{eq:lower_a0}, and \eqref{eq:bound_a13}, we can simplify \eqref{eq:prob_tail} and \eqref{eq:uppp} to conclude that with probability at least 
\begin{align}\label{eq:prob_atleast}
    1-C\left(\gamma^{-5/6} N^{3-c_1\log (1+c_2K)}+ N^{1-c'\gamma K^2}\right),
    \end{align}
    $\sigma_{\max}(X)$ satisfies
    \begin{align*}
     \sigma_{\max}(X)&\leq   \sqrt{\gamma}+1+C_1(\delta +\gamma^{-\frac{3}{2}}q^{-1}) \leq \sqrt{\gamma} +1+C_1'(1+K^{-1}\gamma^{-3/2})\delta\\
   &=\sqrt{\gamma} +1+(K+\gamma^{-3/2})\frac{C_1'\eta}{\sqrt{1\vee \log \eta}} .
\end{align*}
This finishes the proof of \eqref{eq:sigmax_X}. Now we turn to the expectation bound \eqref{eq:Esigma_max}. Since entries in $X$ are bounded by $q^{-1}$, from the concentration of operator norm in \cite[Example 8.7]{boucheron2013concentration}, 
\begin{align}\label{eq:lipschitz}
    \mathbb P \left( \left| \sigma_{\max}(X)-\mathbb E[\sigma_{\max}(X)]\right|\geq \delta \right) \leq 2\exp(-q^2\delta^2/4)\leq 2\exp(-c_2^2K^2/4),
\end{align}
where the last inequality is due to \eqref{eq:qdelta}.
From  \eqref{eq:prob_atleast} and \eqref{eq:lipschitz}, we can take $K=\gamma^{-1/2} K_0$ for an large enough absolute constant $K_0>0$ such that
\begin{align}
  \mathbb P \left( \left| \sigma_{\max}(X)-\mathbb E[\sigma_{\max}(X)]\right|\leq \delta \right)+   \mathbb P \left(\sigma_{\max}(X)\leq \sqrt{\gamma} +1+C_1'(1+K_0^{-1}\gamma^{-1})\delta \right)>1.
\end{align}
This implies the intersection of the two events 
\[ \left\{ \left| \sigma_{\max}(X)-\mathbb E[\sigma_{\max}(X)]\right|\leq \delta\right\} \quad \text{and} \quad  \left\{ \sigma_{\max}(X)\leq \sqrt{\gamma} +1+C_1'(1+K_0^{-1}\gamma^{-1})\delta \right\}
\]
are non-empty. 
Hence for some absolute constants $C_2',C_2>0$,
\begin{align}
    \mathbb E[\sigma_{\max}(X)]\leq \sqrt{\gamma}+1+C_2'\gamma^{-1}\delta=\sqrt{\gamma}+1+\frac{C_2\gamma^{-3/2}\eta}{\sqrt{1\vee \log \eta}}.
\end{align}
This finishes the proof of Theorem \ref{thm:lambdamax}  {by using the assumption $\gamma=\Omega(1)$}.
\end{proof}

Based on Theorem \ref{thm:lambdamax}, we prove Theorem \ref{thm:main_sparse_max}.

\begin{proof}[Proof of Theorem \ref{thm:main_sparse_max}]
Take $q=\sqrt{d}$ and $X=\frac{1}{\sqrt{d}} (A-\mathbb EA)$ in Theorem \ref{thm:lambdamax}.
We have 
\begin{align*}
    |X_{ij}|\leq \frac{1}{\sqrt{d}}, \quad \mathbb E |X_{ij}|^2\leq \frac{\kappa}{N},\quad \text{with} \quad  \kappa=\frac{\max_{ij} p_{ij}}{d/N}.
\end{align*}
Also
\begin{align*}
      \max_{i\in [ {n}]} \sum_{j\in [ {m}]} \mathbb E |X_{ij}|^2 &\leq \frac{1}{d}\max_{i\in [m]} \sum_{j\in [n]}p_{ij}=\rho_{\max},\\
    \max_{j\in [ {m}]}  \sum_{ {i}\in [ {n}]} \mathbb E  |X_{ij}|^2 &\leq \frac{1}{d}\max_{i\in [n]}\sum_{j\in [m]}p_{ij}=\tilde{\rho}_{\max}.
\end{align*}
 Equation \eqref{eq:concentration_A} follows from \eqref{eq:sigmax_X} by taking $K=C_3\gamma^{-3/2}$ for a sufficiently large constant $C_3$, and the probability estimate  can be lower bounded by  $1-C\gamma^{-5/6} N^{-3}$ for some absolute constants $C>0$. The expectation bound in \eqref{eq:expectation_A} follows directly from \eqref{eq:Esigma_max}.
\end{proof}
 \begin{remark}
     The probability bound in the statement of Theorem \ref{thm:main_sparse_max} can be improved to $1-O(\gamma^{-5/6} N^{-a})$ for any constant $a>0$ by taking a larger constant $C_3$ in the proof.
 \end{remark}

\section{Probabilistic bounds on the smallest singular value}\label{sec:prob_smallest}

We now turn to the probabilistic lower bound on the smallest singular values for a general random matrix model.

\begin{proof}[Proof of Theorem \ref{thm:least}]
 
Using Bennett's inequality \cite[Theorem 2.9.2]{vershynin_2018},  we obtain for any $j\in [m]$ and  $t\geq 0$,
\begin{align*}
&\mathbb P\left(\sum_{i} \left(|X_{ij}|^2-\mathbb E|X_{ij}|^2\right)\leq -t \right)
=   \mathbb P \left( \sum_{i} \left(-|X_{ij}|^2+\mathbb E|X_{ij}|^2\right)\geq t \right)\leq \exp \left(-\frac{1}{2} q^2 h(2t) \right).
\end{align*}
Taking $t=\delta \tilde{\rho}_{\min}$ implies for any $j\in [m]$,
\begin{align}
    \mathbb P \left( \sum_{i} |X_{ij}|^2\leq  \tilde\rho_{\min}(1-\delta)\right)\leq \exp\left(-\frac{1}{4}q^2 h\left(2\delta\tilde\rho_{\min}\right)\right).
\end{align}
Combined with \eqref{eq:rhomaxbound} and \eqref{eq:tilderhomaxbound},  after a union bound,    with probability at least 
\begin{align}\label{eq:prob1}
    1-m\exp\left(-\frac{1}{4}q^2 h\left(2\delta \tilde\rho_{\min}\right)\right)-2n\exp\left(-\gamma q^2 h(\delta)\right), 
\end{align}
we have for  {all $j\in [m]$},
\[ \tilde{\rho}_{\min}(1-\delta)\leq  \sum_{i\in [n]}|X_{ij}|^2\leq 1+\delta,
\quad \text{ and }\quad  \max_{i\in [ {n}]} \sum_{j\in [ {m}]} |X_{ij}|^2 \leq \gamma(1+\delta).
\]

 Moreover, from the concentration of spectral norm of $H$ given in \cite[Equations (2.4) and (2.6)]{benaych2020spectral}, for $q\geq \sqrt{\log n}$, with probability at least $1-2\exp(-q^2)$,
$ \|X\|\leq C_4
$
for some  absolute constant $C_4>0$.
From Theorem \ref{thm:rhoB}, with probability 
$1-C\gamma^{-\frac{5}{6}}n^{3-c_1q\log(1+\sqrt{\delta})}$, $\rho(B)\leq \gamma^{\frac{1}{4}}(1+\sqrt{\delta}).$ 
Note that for $x\in (0,2]$, 
\[h(x)=(1+x)\log(1+x)-x\geq  \frac{x^2}{2(1+x/3)}\geq \frac{3}{10}x^2\]  and $\log(1+\sqrt{\delta})\geq \frac{\sqrt{\delta}}{1+\sqrt{\delta}}\geq  \frac{1}{2}\sqrt{\delta}$.

With the assumption that $\tilde{\rho}_{\min}\geq \sqrt{\gamma}$,
conditioned on all events above,   from  Lemma \ref{lem:lowerbound}, we have with probability at least 
\begin{align}\label{eq:prob_h}
    1-3n\exp\left(-\frac{3}{10}\gamma q^2\delta^2\right)-2\exp(-q^2)-C\gamma^{-\frac{5}{6}}n^{3-\frac{1}{2}q\sqrt{\delta}},
\end{align}
\begin{align}\label{eq:sigma_prob_lower_bound}
\sigma_{\min}^2(X)\geq  \frac{\sqrt{\gamma}- \gamma}{\sqrt{\gamma}+\delta}\left( \frac{\beta^2}{
\beta^2+\delta^2}\tilde{\rho}_{\min}-\beta^2-C_3\delta^2- \delta\right)_+,
\end{align} 
    where $\beta=\gamma^{1/4}(1+\sqrt{\delta})$ and $C_3=4 \gamma^{-\frac{1}{2}}(C_4+\gamma^{-1}\delta) \frac{\sqrt{\gamma}+\delta}{\sqrt{\gamma}-\gamma}$.

 {Since $\delta\in (0,1]$ and $\tilde{\rho}_{\min}\leq 1$, \eqref{eq:sigma_prob_lower_bound} implies 
\begin{align}\label{eq:simpler}
    \sigma_{\min}^2(X)\geq \frac{\sqrt{\gamma}- \gamma}{\sqrt{\gamma}+\delta}\left( \tilde{\rho}_{\min} -\sqrt{\gamma} \left(1+3\sqrt{\delta}\right)-\delta-\frac{C_5}{\gamma^2(1-\sqrt \gamma)}\delta^2\right)_+,
\end{align}
where $C_5$ is an absolute constant.
  Using the Assumption~\ref{assumption_4},   \eqref{eq:simpler} implies \eqref{eq:sigma_min}.}
  
 {Next, we consider the expectation bound. Repeating the proof of \cite[Example 8.7]{boucheron2013concentration}, we have the following concentration inequality for $\sigma_{\min}(X)$:
    \begin{align}
        \mathbb P(|\sigma_{\min}(X)-\E[\sigma_{\min}(X)]|\geq \delta^{1/4} )\leq 2\exp(-q^2\delta^{1/2}/4).
    \end{align}
        We can take $q\geq C_0\max \left\{ \delta^{-1/2}, \delta^{-1}\gamma^{-1/2}\sqrt{\log n}\right\}$ for a sufficiently large $C_0$ such that 
    \begin{align*}
        &\mathbb P\left(|\sigma_{\min}(X)-\E[\sigma_{\min}(X)]|\leq  \delta^{1/4}\right)\\
        &+\mathbb P\left( \sigma_{\min}^2(X)\geq \frac{\sqrt{\gamma}- \gamma}{\sqrt{\gamma}+\delta}\left( \tilde{\rho}_{\min} -\sqrt{\gamma} \left(1+3\sqrt{\delta}\right)-\delta-\frac{C_5}{\gamma^2(1-\sqrt \gamma)}\delta^2\right)\right)>1.
    \end{align*}
    This implies the intersection of the event $\left\{ |\sigma_{\min}(X)-\E[\sigma_{\min}(X)]|\leq  \delta^{1/4}\right\}$ and \eqref{eq:simpler} is nonempty. Therefore, under Assumption~\ref{assumption_4}, we obtain
    \[ \mathbb E[\sigma_{\min}(X)]\geq \sqrt{(1-\sqrt{\gamma})(\tilde{\rho}_{\min}-\sqrt{\gamma})}-O(\delta^{1/4}). \]
    This finished the proof of \eqref{eq:Esigma_min}.}
\end{proof}

\begin{proof}[Proof of Theorem \ref{thm:expectation_A_lowerbound}]
 {We take $q=\sqrt{d}$ in Theorem \ref{thm:least}. Then, under Assumption~\ref{assumption_2}, with probability at least
$1-O(n^{-3})$, \eqref{eq:smin} holds. \eqref{eq:sminE} follows directly from Theorem \ref{thm:least}.}
\end{proof}

With Theorems \ref{thm:main_sparse_max} and  \ref{thm:expectation_A_lowerbound}, we prove Corollary \ref{cor:no_outliers}.

\begin{proof}[Proof of Corollary \ref{cor:no_outliers}]

From the assumption \eqref{eq:homo_assumption}, 
Theorem \ref{thm:main_sparse_max} implies  with probability $1-O(n^{-3})$,
\begin{align}\label{eq:A1}
    \frac{1}{\sqrt{d}} \sigma_{\max}(A-\mathbb EA)\leq 1+\sqrt{y}+o(1).
\end{align}
From  \eqref{eq:homo_assumption},  $\tilde{\rho}_{\min}=1+o(1)$. Taking $\delta^2= \frac{\log n}{d}$  in Theorem \ref{thm:expectation_A_lowerbound}, we obtain  with probability $1-O(n^{-3})$, 
\begin{align}\label{eq:A2}
\frac{1}{\sqrt d}\sigma_{\min}(A-\mathbb EA)\geq 1-\sqrt{y}-o(1).
\end{align}

We can apply the proof of  \cite[Corollary 4.3 and Theorem 8.2]{zhu2020graphon} to inhomogeneous Erd\H{o}s-R\'{e}nyi bipartite graphs. One can show in the same way that, almost surely, the empirical spectral distribution of $\frac{1}{d} (A-\mathbb EA)^\top (A-\mathbb EA)$ converges to the Mar\v{c}enko-Pastur law supported on the interval  $[(1-\sqrt y)^2, (1+\sqrt y)^2]$. Therefore almost surely, 
\begin{align} \label{eq:law}
\frac{1}{\sqrt{d}} \sigma_{\max}(A-\mathbb EA)\geq 1+\sqrt{y}-o(1),  \quad  \frac{1}{\sqrt d}\sigma_{\min}(A-\mathbb EA)\leq 1-\sqrt{y}+o(1).
\end{align}
From \eqref{eq:A1}, \eqref{eq:A2}, and \eqref{eq:law}, the convergence results in  \eqref{eq:upper_Bai} hold.
\end{proof}

\subsection*{Acknowledgments}
We are grateful to Antti Knowles and Ramon van Handel for helpful discussions and to the latter for a detailed explanation of the work \cite{brailovskaya2022universality}.  Y.Z. thanks Ludovic Stephan for a careful reading of this paper and many useful suggestions. We thank the editor and anonymous referees for helping us improve the presentation.

I.D. and Y.Z. acknowledge support from NSF  DMS-1928930 during their participation in the program Universality and Integrability in Random Matrix Theory and Interacting Particle Systems hosted by the Mathematical Sciences Research Institute in Berkeley, California, during the Fall semester of 2021. I.D. is partially supported by NSF DMS-2154099. Y.Z. is partially supported by  NSF-Simons Research Collaborations on the Mathematical and Scientific Foundations of Deep Learning.  This work was done in part while Y.Z. was visiting the Simons Institute for the Theory of Computing in the Fall of 2022.

\bibliographystyle{plain}
\bibliography{ref.bib}
\end{document}